\documentclass[12pt]{amsart}

\usepackage[hmargin=0.8in,height=8.6in]{geometry}
\usepackage{amssymb,amsthm}
\usepackage{delarray,verbatim}
\usepackage{ifpdf}
\ifpdf
\usepackage[pdftex]{graphicx}
\else
\usepackage[dvips]{graphicx}
\fi

\usepackage{ifpdf}
\usepackage{color}
\definecolor{webgreen}{rgb}{0,.5,0}
\definecolor{webbrown}{rgb}{.8,0,0}
\definecolor{emphcolor}{rgb}{0.95,0.95,0.95}

\renewcommand{\theequation}{\thesection.\arabic{equation}}
\numberwithin{equation}{section} 

\newtheorem {thm}{Theorem}[section]
\newtheorem {prop}{Proposition}[section]
\newtheorem {lemm}{Lemma}[section]

\newtheorem*{assumption}{Assumption}

\theoremstyle{remark}
\newtheorem {rem}{Remark}[section]

 \DeclareMathOperator{\argmax}{argmax}

\newcommand{\R}{\mathbb R}
\newcommand{\N}{\mathbb N}
\newcommand{\Z}{\mathbb Z}
\newcommand{\PP}{\mathbb P}
\renewcommand{\P}{\mathbb P}
\newcommand{\A}{\mathcal A}
\newcommand{\cZ}{\mathcal Z}
\newcommand{\cP}{\check P}
\newcommand{\E}{\mathbb E}

\newcommand{\M}{\mathcal M}

\newcommand{\U}{\mathcal U}

\newcommand{\F}{\mathbb F}

\renewcommand{\bar}{\overline}

\newcommand{\s}{\mathcal{S}}
\renewcommand{\emptyset}{\varnothing}
\newcommand{\Fc}{\mathcal F}
\newcommand{\vP}{\vec{\Pi} }
\newcommand{\vp}{\vec{\pi} }
\newcommand{\vz}{\vec{z} }
\newcommand{\vx}{\vec{x} }
\newcommand{\eps}{\varepsilon}
\newcommand{\e}{\mathrm{e}}
\newcommand{\G}{\mathcal{G}}
\renewcommand{\S}{\mathcal{S}}

\title[Inventory Management with Partially Observed Nonstationary Demand ]{Inventory Management with Partially Observed Nonstationary Demand}

\author{Erhan Bayraktar }
\address[E. Bayraktar]{Department of
  Mathematics, University of Michigan, Ann Arbor, MI 48109}
\email{erhan@umich.edu}
\thanks{E. Bayraktar is supported in part by the National Science Foundation, under grant DMS-0604491. }

\author{Michael Ludkovski}
\address[M.\ Ludkovski]{Department of Statistics and Applied Probability, University of California Santa Barbara, CA 93106-3110}
\email{ludkovski@pstat.ucsb.edu}


\begin{document}
\begin{abstract}
We consider a continuous-time model for inventory management with Markov modulated
non-stationary demands. We introduce active learning by assuming that the state of the world is
unobserved and must be inferred by the manager. We also assume that demands are observed only
when they are completely met. We first derive the explicit filtering equations and pass to an
equivalent fully observed impulse control problem in terms of the sufficient statistics, the a
posteriori probability process and the current inventory level. We then solve this equivalent
formulation and directly characterize an optimal inventory policy. We also describe a
computational procedure to calculate the value function and the optimal policy and present two
numerical illustrations.
\end{abstract}

\maketitle
\section{Introduction}\label{sec:intro}
Inventory management aims to control the supply ordering of a firm so that inventory costs are
minimized and a maximum of customer orders are filled. These are competing objectives since low
inventory stock reduces storage and ordering costs, whereas high inventory avoids stock-outs.
The problem is complicated by the fact that real-life demand is never stationary and therefore
the inventory policy should be non-stationary as well. A popular method of addressing this
issue is to introduce a regime-switching (or Markov-modulated) demand model. The regime is
meant to represent the economic environment faced by the firm and drives the parameters
(frequency and size distribution) of the actual demand process. However, typically this
economic environment is unknown. From a modeling perspective, this leads to a \emph{partially
observed} hidden Markov model for demand. Thus, the inventory manager must simultaneously learn
the current environment (based on incoming order information), and adjust her inventory policy
accordingly in anticipation of future orders.

The literature on inventory management with non-stationary Markovian demand originated with
\cite{SongZipkin} who considered a continuous-time model where the demand \emph{levels} or
intensity are modulated by the state of the world, which are assumed to be observable. The
discrete-time counterpart of this model was then analyzed by \cite{SethiCheng97} who allowed a
very general cost structure and proved a more formal verification theorem for existence and
regularity of the value function and existence of an optimal feedback policy. More recent work
on \emph{fully observed} non-stationary demand can be found in \cite{BensoussanLiuSethi05}.

Inventory management with {partial information} is a classical topic of operations research. In
the simplest version as analyzed by \cite{Azoury85,Lovejoy90} and references therein, the
demand distribution is unknown and must be \emph{learned} by the controller. Thus, a
finite-horizon discrete time {parameter adaptive} model is considered, so that the demand
distribution is taken to be stationary and i.i.d., but with an unknown parameter. This
parameter is then inferred over time using either an exponential smoothing or a Bayesian
learning mechanism. In these papers the method of solution relied on the special structure of
some particular cases (e.g.\ uniform demand level on $[0,w]$, with $w$ unknown) where a
dimension reduction is possible, so that the learning update is simplified. Even after that,
the problem remains extremely computationally challenging; accordingly the focus of
\cite{Lovejoy90} has been on studying approximate myopic or limited look-ahead policies.

Another important strand of literature investigates the \emph{lost sales} uncertainty. In that
case, demand is observed only if it is completely met; instead of a back-order, unmet demand is
lost. This then creates a partial information problem if demand levels are non-stationary. In
particular, demand levels evolving according to a discrete-time Markov chain have been
considered in the newsvendor context (completely perishable inventory) by
\cite{LarivierePorteus,BeyerSethi05,BensoussanMOR07}, and in the general inventory management
case by \cite{BensoussanMinjarez08}. \cite{BensoussanSethiCR05,BensoussanEtalSICON07} have
analyzed the related case whereby current inventory level itself is uncertain.

The main inspiration for our model is the work of \cite{TreharneSox} who considered a version
of the \cite{SethiCheng97} model but under the assumption that current world state is
\emph{unknown}. The controller must therefore filter the present demand distribution to obtain
information on the core process. \cite{TreharneSox} take a partially observed Markov decision
processes (POMDP) formulation; since this is computationally challenging, they focus on
empirical study of approximate schemes, in particular myopic and limited look-ahead learning
policies, as well as open-loop feedback learning. A related problem of \emph{dynamic pricing}
with unobserved Markov-modulated demand levels has been recently considered by
\cite{AvivPazgal05}. In that paper, the authors also work in the POMDP framework and propose a
different approximation scheme based on the technique of information structure modification.

In this paper we consider a continuous-time model for inventory management with Markov
modulated non-stationary demands. We take the \cite{SongZipkin} model as the starting point and
introduce active learning by assuming that the state of the core process is unobserved and must
be inferred by the controller. We will also assume that the demand is observed only when it is
completely met, otherwise censoring occurs. Our work extends the results of
\cite{BensoussanMOR07,BensoussanMinjarez08} and \cite{TreharneSox} to the continuous-review
asynchronous setting. Use of continuous- rather than discrete-time model facilitates some of
our analysis. It is also more realistic in medium-volume problems with many asynchronous orders
(e.g.\ computer hardware parts, industrial commodities, etc.), especially in applications with
strong business cycles where demand distribution is highly non-stationary. In a continuous-time
model, the controller may adjust inventory immediately after a new order, but also between
orders. This is because the controller's beliefs about the demand environment are constantly
evolving. Such qualitative distinction is not possible with discrete epochs where controls and
demand observations are intrinsically paired.

Our method of solution consists of two stages. In the first stage (Section~\ref{sec:probStat}),
we derive explicit filtering equations for the evolution of the conditional distribution of the
core process. This is non-trivial in our model where censoring links the observed demand
process with the chosen control strategy. We use the theory of partially observed point
processes to extend earlier results of \cite{Arjas92,LS07,baylud08} in Proposition~\ref{cor:pdp}.
The piecewise deterministic strong Markov process obtained in Proposition~\ref{cor:pdp} allows us
then to give a simplified and complete proof of the dynamic programming equations, and describe
an optimal policy (see Section~\ref{sec:sequential}). We achieve this by leveraging the general
probabilistic arguments for the impulse control of piecewise deterministic processes of
\cite{davis93,CostaDavis89,MR1150206} and using direct arguments to establish necessary
properties of the value function. Our approach is in contrast with the POMDP and
quasi-variational formulations in the aforementioned literature that make use of more analytic
tools.

Our framework also leads to a different flavor for the numerical algorithm. The closed-form
formulas obtained in Sections \ref{sec:probStat} and \ref{sec:sequential} permit us to give a
direct and simple-to-implement computational scheme that yields a precise solution. Thus, we do
not employ any of the approximate policies proposed in the literature, while maintaining a
competitive computational complexity.

To summarize, our contribution is a full and integrated analysis of such incomplete information
setting, including derivation of the filtering equations, characterization of an optimal policy
and a computationally efficient numerical algorithm. Our results show the feasibility of using
continuous-time partially observed models in inventory management problems. Moreover, our model
allows for many custom formulations, including arbitrary ordering/storage/stock-out/salvage
cost structures, different censoring formulations, perishable inventory and supply-size
constraints.

The rest of the paper is organized as follows. In the rest of the introduction we will give an
informal description of the inventory management problem we are considering. In
Section~\ref{sec:probStat} we make the formulation more precise and show that the original
problem is equivalent to a fully observed impulse control problem described in terms of
sufficient statistics. Moreover, we characterize the evolution of the paths of the sufficient
statistics. In Section~\ref{sec:sequential} we show that the value function is the unique fixed
point of a functional operator (defined in terms of an optimal stopping problem) and that it is
continuous in all of its variables. Using the continuity of the value function we describe an
optimal policy in Section~\ref{sec:optimal-stragegy}. In Section~\ref{sec:examples} we describe
how to compute the value function using an alternative characterization. Finally, in Section
\ref{sec:illust} we present two numerical illustrations. Some of the longer proofs are left to
the Appendix.

\subsection{Model Description}
In this section we give an informal description of the inventory management with partial
information and the objective of the controller. A rigorous construction is in Section
\ref{sec:probStat}.

Let $M$ be the unobservable Markov core process for the economic environment. We assume that
$M$ is a continuous-time Markov chain with finite state space $E \triangleq \{ 1, \ldots, m\}$.
The core process $M$ modulates customer orders modeled as a compound Poisson process $X$. More
precisely, let $Y_1, Y_2, \ldots$ be the consecutive order sizes, taking place at times
$\sigma_1, \sigma_2, \ldots$. Define
\begin{align}
N(t) & = \sup\{ s : \sigma_s < t \}
\end{align}
to be the number of orders received by time $t$, and
\begin{align}
X_t = \sum_{k=1}^{N(t)} Y_k,
\end{align}
to be the total order size by time $t$. Then, the intensity of $N$ (and $X$) is $\lambda_i$
whenever $M$ is at state $i$, for $i \in E $. Similarly, the distribution of $Y_k$ is $\nu_i$
conditional on $M_{\sigma_k} = i$. This structure is illustrated schematically in Figure
\ref{fig:M-X-diagram}. The assumption of demand following a compound Poisson process is
standard in the OR literature, especially when considering large items (the case of demand
level following a jump-diffusion is investigated by \cite{BensoussanLiuSethi05}).

\begin{figure}%[ht]\hspace{-1.5in}
\centering{\includegraphics*[height=4in,width=4in,clip]{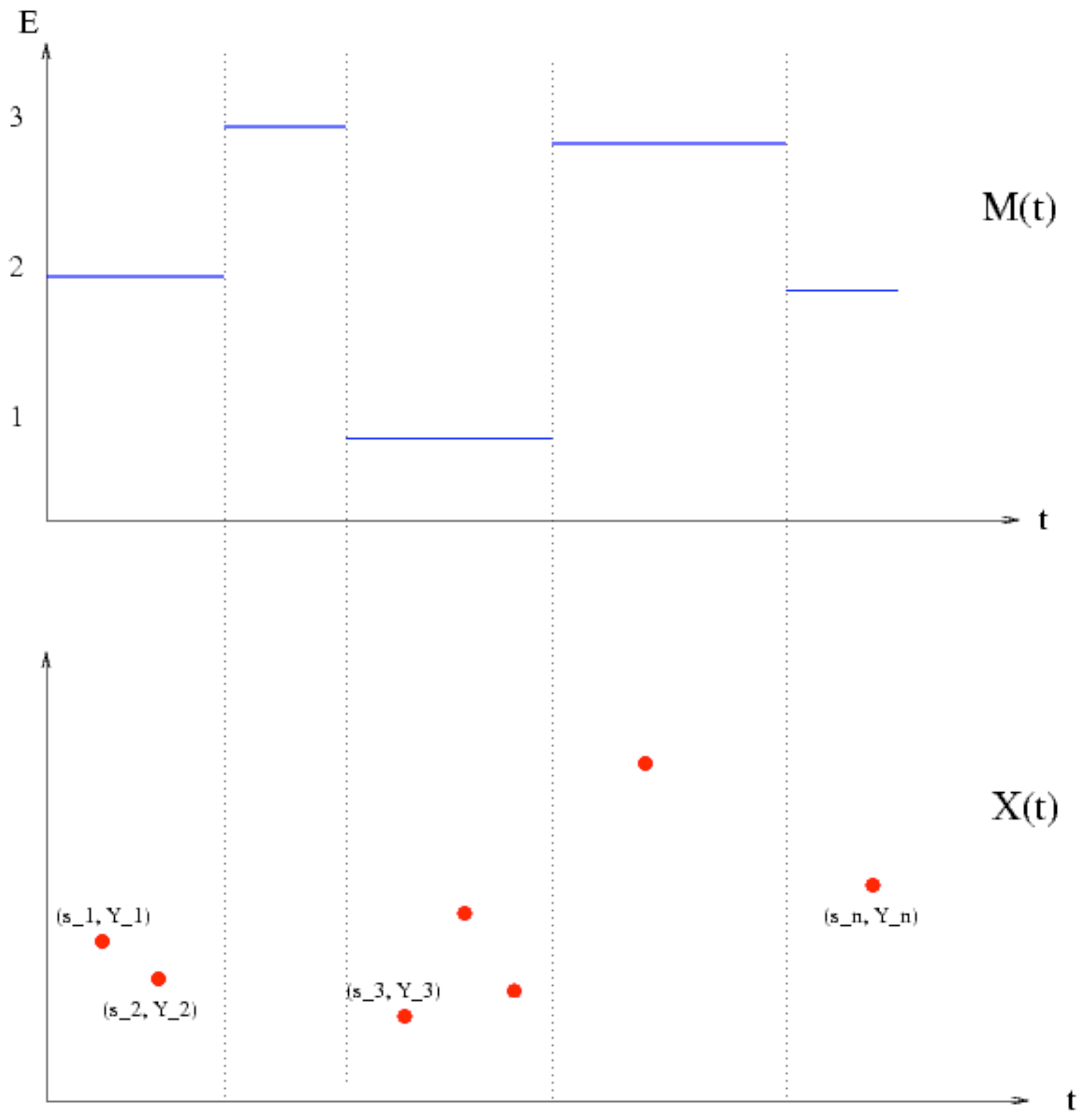}} \caption{The
environment process $M$ modulates the marked point process $X = ( (\sigma_1, Y_1), \ldots)$
representing demand order flow. In this illustration, $M$ takes the values $M_t \in \{1,2,3\} =
E$, and the mark distributions are arranged such that $\mu_1 < \mu_2 < \mu_3$ where $\mu_i =
\E[ Y | M_t = i]$. Also, the intensities are arranged $\lambda_3 < \lambda_2 < \lambda_1$, with
$\lambda_i = \E[ \sigma_1 | M_u = i, 0 \le u \le T]$. Thus, in regime 1 demand is frequent and
demand amounts are small, in regime 2 demand is average and amounts are small to average, and
in regime 3 demand is rare but consists of relatively large orders.\label{fig:M-X-diagram}}
\end{figure}

We assume that orders are integer-sized with a fixed upper bound $R$, namely
\begin{assumption}
Each $\nu_i$, $i \in E$ is a discrete bounded distribution on $\mathbb{Z}$, so that $Y_k \in
\{1, 2, \ldots, R\}$.
\end{assumption}

The controller cannot observe $M$; neither does she directly see $X$. Instead, she receives the
\emph{censored} order flow $W$. The order flow $W$ consists of filled order amounts $(Z_\ell)$,
which informally correspond to the minimum between actual order size and available inventory.
Hence, if total inventory is insufficient, a stock-out occurs and the excess order size is left
indeterminate. The model where order flow is always observed will also be considered as a
special case with zero censoring. The precise description of $W$ will be given in Section 2.1.

Let $P_t$ be the current inventory at time $t$. We assume that inventory has finite stock
capacity, so that $P_t \in [0, \bar{P}]$. Inventory changes are driven by two variables: filled
customer orders described by $W$ and supply orders. Customer orders are assumed to be
exogenous; every order is immediately filled to its maximum extent. When a stock-out occurs, we
consider two scenarios:
\begin{itemize}
\item If there is no censoring, then the excess order amount is immediately back-ordered at a
higher penalty rate.

\item If there is censoring, a lost opportunity cost is assessed in proportion to expected
excess order amount.
\end{itemize}
Otherwise, the inventory is immediately decreased by the order amount. Supply orders are
completely at the discretion of the manager. Let $\xi_1, \xi_2, \ldots$ denote the supply
amounts (without back-orders) and $\tau_1, \tau_2, \ldots$ the supply times (when supply order
is placed).

To summarize, the dynamics of $P$ are (compare to \eqref{eq:rel-pi-x} below):
\begin{align}\label{eq:P-dynamics} \left\{ \begin{array}{lrr}
P_{\sigma_k} &= (P_{\sigma_k -} - Z_k), & \text{fill customer order} \\
P_{\tau_k} &= P_{\tau_k-} + \xi_k, & \text{new supply order}\\
dP_t &= 0 &\text{otherwise}. \end{array} \right.
\end{align}

We assume that the manager can only increase inventory (no disposal is possible) and that
inventory never perishes. Alternatives can be straightforwardly dealt with and are considered
in a numerical example in Section~\ref{sec:examples}.

We denote the entire inventory strategy by the right-continuous piecewise constant process $\xi
\colon [0,T] \times \Omega \to \mathcal{A}$, with $\xi_t = \xi_k$ if $ \tau_{k} \leqslant t <
\tau_{k+1}$ or
\begin{align}\label{eq:xi}
\xi_t = \sum_{\tau_{k} \leq  T} \xi_k \cdot 1_{[\tau_k, \tau_{k+1})}(t).
\end{align}
The goal of the inventory manager is to minimize inventory costs as appearing in the objective
function \eqref{eq:objective} below among all \emph{admissible} inventory strategies $\xi$. The
admissibility condition concerns first and foremost the information set available to the
manager. Because only $W$ is observable, the strategy $\xi$ should be determined by the
information generated by $W$, namely each $\tau_k$ must be a stopping time of the filtration
$\F^W \triangleq\{\mathcal{F}^W_t\}$ of $W$. Similarly, the value of each $\xi_k$ is determined
by the information $\Fc^W_{\tau_k}$ revealed by $W$ until $\tau_k$. Also, without loss of
generality we assume that $\xi$ has a finite number of actions, so that $\P(\tau_k < T) \to 0$
as $k\to \infty$. Since strategies with infinitely many actions will have infinite costs, we
can safely exclude them from consideration. We denote by $\U(T)$ the set of all such
\emph{admissible strategies} on a time interval $[0,T]$.

The selected policy $\xi$ directly influences current stocks $P$; when we wish to emphasize
this dependence we will write $P_t \equiv P_t^\xi$. The cost of implementing $\xi$ is as
follows. First, inventory storage costs at fixed rate $c(P_t)$ are accessed. We assume that $c$
is positive, increasing and continuous. Second,  a supply order of size $\xi_k$ costs $h\cdot
\xi_k+\zeta$, for positive $h$ and $\zeta$. Finally, if a stock-out occurs due to insufficient
existing stock, a penalty reflecting the lost opportunity cost is assessed at amount $\E\left[
K( (Y_k - P_{\sigma_k-})_+ ) | \Fc^W_t \right]$. We assume that the penalty function $K$ is
positive and increasing with $K(0) = 0$. Thus, the total performance of a strategy $\xi$ on the
horizon $[0,T]$ is
\begin{align}\label{eq:objective}
\int_0^T \e^{- \rho t }   c(P_t) \, dt + \sum_{k : \tau_k < T} \e^{-\rho \tau_k}(h \cdot \xi_k
+ \zeta) + \sum_{\ell : \sigma_\ell < T} \e^{- \rho \sigma_\ell } K(
(Y_\ell-P_{\sigma_\ell-})_+ ),
\end{align}
where $\rho \ge 0$ is the discount factor for future revenue. Table \ref{table:params}
summarizes our notation and the meaning of the model parameters.

\begin{table}
\begin{tabular}{|c|l|} \hline
Parameter & Meaning \\ \hline $c(a)$ & storage cost for $a$ items per unit time \\
$K(a)$ & instantaneous stock-out cost for a shortage of $a$ items \\
$h $ &  order cost for one item \\
$\zeta$ & fixed cost of placing an order \\
$\rho$ & discount factor for NPV calculation \\
$\bar{P}$ & maximum inventory size \\
$R$ & maximum demand size \\
\hline
\end{tabular}
\medskip \caption{Parameters of the model\label{table:params}}
\end{table}

%Other possible modifications to consider are (all of these have been mentioned in existing
%literature):
%\begin{enumerate}
%\item Perishable storage: the stock exogenously decreases according to the deterministic
%equation $dP_t = - f(P_t) \, dt$ (eg food items). This would require modeling stock as a
%\emph{continuous} quantity.
%
%\item Stochastic lead times (eg drawn from a stationary independent distribution).
%
%\item Lost sales: instead of stockout, orders that cannot be filled are unobserved. This
%introduces censoring of $Y$'s. In that case, one can still look at $K( (Y_k -
%P_{\sigma_k-})_+)$ as representing opportunity costs due to lost sales.
%
%\item Order costs/stockout costs could depend on $M$. For example, if $M$ represents economy
%state and economy is robust, opportunity cost of lost sales will be higher.
%
%\item Inventory can be both increased or decreased, i.e. allowing disposal of unwanted stock.
%
%\item Seasonality in eg demand distributions ($\nu_i$ depending on time) or ordering/storage
%costs.
%\end{enumerate}
%
 %
% Strategies with infinite number of impulses are ruled out automatically due to \xi_k > 1 and finite capacity

Observe that the objective in \eqref{eq:objective} involves the distribution of $Z_k$'s (which
affect $P$) and $Y_k$'s (which enter into objective function through stock-out costs). Both of
these depend on the state $M$. Since the core process $M$ is unobserved, the controller must
therefore carry out a filtering procedure. We postulate that she collects information about $M$
via a Bayesian framework. Let $\vp = (\pi_1, \ldots, \pi_m) \triangleq \left( \P\{M_0 = 1\},
\ldots, \P\{M_0 = m\} \right)$ be the initial (prior) beliefs of the controller about $M$ and
$\P^{\vp}$ the corresponding conditional probability law. The controller starts with beliefs
$\pi$, observes $Z$, updates her beliefs and adjusts her inventory policy accordingly.

These notions and the precise updating mechanism will be formalized in Section \ref{sec:vP}.
The solution will then proceed in two steps: an initial filtering step and a second
optimization step. The inference step is studied in Section \ref{sec:probStat}, where we
introduce the a posteriori probability process $\vP$. The process $\vP$ summarizes the dynamic
updating of controller's beliefs about the Markov chain $M$ given her point process
observations. The optimal switching problem \eqref{def:U} is then analyzed in Section
\ref{sec:sequential}.

\section{Problem Statement}\label{sec:probStat}
In this section we give a mathematically precise description of the problem and show that this problem is equivalent to a
fully observed impulse control problem in terms of the strong Markov process $(\vP,P)$. We will also describe the dynamics of the sufficient statistics $(\vP,P)$.

\subsection{Core Process}
Let $(\Omega, \mathcal{H}, \P)$ be a probability space hosting two independent elements: (i)
a continuous time Markov process $M$ taking values in a finite set $E$, and with infinitesimal
generator $Q=(q_{ij})_{i,j \in E }$, (ii) $X^{(1)} , \ldots, X^{(m)} $ which are  independent
compound Poisson processes with intensities and discrete jump size distributions  $ (\lambda_1,
\nu_1), \ldots , (\lambda_m , \nu_m)$, respectively, $m \in E$.

The core point process $X$ is given by
\begin{align} \label{def:X} X_t \triangleq X_0 +
\int_0^t  \sum_{ i \in E} 1_{ \{ M_s =i \} } \, dX^{(i)}_s , \quad t\ge 0.
\end{align}
Thus, $X$ is a Markov-modulated Poisson process (see e.g. \cite{KarlinTaylor2}); by construction, $X$ has
independent increments conditioned on $M = \{ M_t\}_{t \ge 0}$. Denote by $\sigma_0, \sigma_1 ,
\ldots$ the arrival times of the process $X$,
\begin{align*}
\sigma_\ell \triangleq \inf \{ t > \sigma_{\ell-1} : X_t \ne X_{t- }\}, \qquad \ell \ge 1
\qquad \text{with $\sigma_0 \equiv 0$.}
\end{align*}
and by $Y_1, Y_2 ,\ldots$ the $\R$-valued marks (demand sizes) observed at these arrival times:
\begin{align*}
Y_\ell = X_{\sigma_\ell } - X_{\sigma_\ell- }, \qquad \ell \ge 1.
\end{align*}
Then conditioned on $\{M_{\sigma_\ell} = i\}$, the distribution of $Y_\ell$ is described by $\nu_i(dy)=f_i(y)dy$
on $( \Z_+, \mathcal{B}(\Z_+)) $.

\subsection{Observation Process}
Starting with the marked point process $(M,X)$, the observable is a point process $W$ which is
derived from $(M,X)$. This means that the marks of $W$ are completely determined by $(M,X)$
(and the control). Fix an initial stock $\cP^{(1)}_0 = a$. We first construct an auxiliary
process $\check{W}^{(1)}$. The first mark of $\check{W}^{(1)}$ is $(\sigma_1, Z_1)$, where
$\sigma_1$ is the first arrival time of $X$ and where the distribution of the first jump size
$Z_1 = (Z^1_1, Z^2_1)\in \cZ$ conditional on $(M,X)$ given by
$$ \P ( Z_1 = \vz | Y_1 = y, M_{\sigma_1} = i) = \mathcal{C}(y, \vz; {\cP^{(1)}_{\sigma_1-}}),$$
where $\mathcal{C}(y,\vz; a)$, $y \in \Z_+, \vz \in \cZ, a \in [0,\bar{P}]$ are the
$\{0,1\}$-valued \emph{censoring functions} satisfying $\sum_{\vz \in \cZ} \mathcal{C}(y,\vz;a)
\in \{0,1\}$. In our context, it is convenient to take $\cZ = \Z_+ \times \{ 0, \Delta\}$,
where $\vz \equiv (z^1, z^2) \in \cZ$ represents a filled order of size $z^1$, and the second
component $z^2$ indicates whether a stock-out occurred or not. Censoring of excess orders then
corresponds to $\mathcal{C}(y,\vz; p) = 1_{\{(0,p]\}}(y) 1_{\{(y,0)\}}(\vz) + 1_{\{(p, R)\}}(y)
1_{\{(\lfloor p \rfloor, \Delta)\}}(\vz)$. Alternatively, without censoring we take $\cZ = \Z_+
\times \Z_+$ (second component now indicating actual order size), and $\mathcal{C}(y,\vz; p) =
1_{\{(0,p]\}}(y)1_{\{(y,y)\}}(\vz) + 1_{\{(p,R]\}}(y)1_{\{(\lfloor p \rfloor, y)\}}(\vz)$.

Once $(\sigma_1, Z_1)$ is observed, we update $\cP^{(1)}_{\sigma_1} = a - Z^1_1 \ge 0$, and set
$\cP^{(1)}_{t} = \cP^{(1)}_{\sigma_1}$ for $\sigma_1 \le t < \sigma_2$, where $\sigma_2$ is the
second arrival time of $X$. Proceeding as before, we will obtain the marked point process
$\check{W}^{(1)} = (\sigma_\ell, Z_\ell)$ and the corresponding uncontrolled stock process
$\cP^{(1)}$.

We now introduce the first impulse control. Let $\{\mathcal{F}^{\check{W}^{(1)}}_t\}$ be the
filtration generated by $\check{W}^{(1)}$, and let $(\tau_1, \xi_1)$ be an
$\mathcal{F}^{\check{W}^{(1)}}$-stopping time, and an
$\mathcal{F}^{\check{W}^{(1)}}_{\tau_1}$-measurable $\Z_+$-valued random variable,
respectively, satisfying $\tau_1 \le T, 0 \le \xi_1 \le \bar{P} - \cP^{(1)}_{\tau_1-}$. The
impulse control means that we take $W = \check{W}^{(1)} 1_{[0,\tau_1)}$, $P_t = \cP^{(1)}_t
1_{[0,\tau_1)}$ and repeat the above construction on the interval $[\tau_1, T]$ starting with
the updated value $\cP^{(2)}_{\tau_1} = \cP^{(1)}_{\tau_1-}+\xi_1$.

Inductively this provides the auxiliary point processes $\check{W}^{(k)}$ together with the
impulse controls $(\tau_k, \xi_k)$, $k=1,2,\ldots$. Letting $\xi = (\tau_1, \xi_1, \tau_2,
\xi_2, \ldots)$ we finally obtain the $\xi$-controlled inventory process $P$, as well as the
$\xi$-controlled marked point process $W = \sum_k \check{W}^{(k)} 1_{[0,\tau_k)}$. By
construction, both $P$ and $\xi$ are $\F^W$-measurable. We denote by $\P^{a,\xi}$ the resulting
probability law of $(W,P)$.

Summarizing, $P$ is a piecewise-deterministic controlled process taking values in $[0,
\bar{P}]$ and evolving as in \eqref{eq:P-dynamics}; the arrival times of $W$ are those of $X$,
and the distribution of its marks $(Z_\ell)$ depends inductively on the latest
$P_{\sigma_\ell-}$, the mark $Y_\ell$ of core process $X$, and the censoring functions
$c^p_i(y,z)$. For further details of the above construction of $\P^{a,\xi}$ we refer the
reader to \cite[pp. 228-230]{davis93}. Our use of censoring functions is similar to the
construction in \cite{Arjas92}.

\subsection{Statement of the Objective Function in Terms of Sufficient Statistics.}
Let $D \triangleq \{ \vp \in [0,1]^m \colon \pi_1 + \ldots + \pi_m =1 \}$ be the space of prior
distributions of the Markov process $M$. Also, let $\s(s) = \{ \tau \colon \F-\text{stopping
time}, \tau \le s, \P-\text{a.s} \}$ denote the set of all $\F$-stopping times smaller than or
equal to $s$.

Let $\P^{\vp,a,\xi}$ denote the probability measure $\mathbb{P}^{a,\xi}$ such that the process $M$ has initial distribution $\vp$. That is,
\begin{equation}
\label{def:P-pi}
\P^{\vp,a,\xi} \{ A \} = \pi_1 \, \mathbb{P}^{a,\xi} \{A | M_0=1\} + \ldots
+ \pi_n \, \mathbb{P}^{a,\xi} \{A | M_0=n\}
\end{equation}
for all $A \in \mathcal{F}_T^{W}$. $\P^{\vp}$ can be similarly defined. In the sequel, when $\xi \equiv 0$, we will denote the corresponding probability measure by $\P^{\vp,a}$

We define the $D$-valued \emph{conditional probability process} $\vP(t) \triangleq \left(
\Pi_1(t), \ldots , \Pi_m(t) \right)$ such that
\begin{align}
\label{def:Pi-i}
 \Pi_i(t) = \P^{\vp,a,\xi} \{  M_t =i | \Fc^W_t \}, \quad \text{for $i \in E$, and $t \ge 0$}.
\end{align}
Each component of $\vP$ gives the conditional probability that the current state of $M$ is $\{
i\}$ given the information generated by $W$ until the current time $t$. %Using the process $\vP$

Using $\vP$ we convert the original objective \eqref{eq:objective} into the following
$\F$-adapted formulation. Observe that given $\vP(\sigma_\ell)$, the distribution of $Y_\ell$
is $\sum_{i\in E} \Pi_i(\sigma_\ell-)\nu_i$. Therefore, starting with initial inventory $a
\in [0, \bar{P}]$ and beliefs $\vp$, the performance of a given policy $\xi \in \U(T)$ is
\begin{equation}\label{eq:policy-performance}
\begin{split}
J^{\xi}(T,\vp,a) &\triangleq \E^{\vp,a,\xi} \bigg[  \int_0^{T} \e^{- \rho t }   c(P_t) \, dt +
\sum_{\ell \in \N_+} \e^{- \rho \sigma_\ell } \sum_{i \in E} \Pi_i(\sigma_\ell)
\int_{\R_+} K \left( (y -P_{\sigma_\ell-})_+ \right) \nu_i(dy)
\\& \qquad + \sum_{k \in \N_+} \e^{-\rho \tau_k}(h \xi_k+\zeta)
\bigg].
\end{split}
\end{equation}
The first argument in $J^\xi$ is the remaining time to maturity. Also,
\eqref{eq:policy-performance} assumed that the terminal salvage value is zero, so at $T$ the
remaining inventory is completely forfeited. The inventory optimization problem is to compute
\begin{align}
\label{def:U} U (T, \vp, a) \triangleq \inf_{ \xi \in \U(T)} J^{\xi}(T,\vp,a),
\end{align}
and, if it exists, find an admissible strategy $\xi^*$ attaining this value. Without loss of
generality we will restrict the set of admissible strategies satisfying
\begin{equation}\label{eq:admissiblity}
\E^{\vp,a,\xi}\left[ \sum_{k \in \N_+}
\e^{-\rho \tau_k} (h \cdot \xi_k+\zeta)\right] <\infty
\end{equation}
otherwise infinite costs would be incurred. Note that the admisible strategies have
 finitely many
switches almost surely for any given path. The
equivalence between the ``separated'' value function in \eqref{def:U} and the original setting
of \eqref{eq:objective} is standard in Markovian impulse control problems, see e.g.\
\cite{BertsekasBookVol1}.

The following notation will be used in the subsequent analysis:
\begin{equation}\label{eq:defI}
I(t) \triangleq \int_0^{t} \sum_{i \in E} \lambda_i 1_{\{M_s=i\}}ds,
\end{equation}
and
\begin{equation}\label{eq:defblamb}
\bar{\lambda} \triangleq \max_{i \in E}\lambda_i.
\end{equation}
It is worth noting that the probability of no events for the next $u$ time units is $\P^{\vp} \{ \sigma_1
> u \} = \E^{\vp}[ \e^{-I(u)}]$.

\subsection{Sample paths of $(\vP,P)$.}\label{sec:vP}
In this section we describe the filtering procedure of the controller. In particular, Theorem
\ref{cor:pdp} explicitly shows the evolution of the processes $(\vP,P)$. This is non-trivial in
our model where censoring links the observed demand process with the chosen control strategy.
The description of paths of the conditional probability process when the control does not alter
the observations is discussed in Proposition 2.1 in \cite{LS07} and Proposition 2.1 of
\cite{BS06}. Filtering with point process observations has also been studied by
\cite{Arjas92,ElliottMalcolm04,ElliottMalcolm05,AllamDufourBertrand01}. Also, see
\cite{ElliottBook} for general description of inference in various hidden Markov models in
discrete time.

 Even
though the original order process $X$ has conditionally independent increments, this is no
longer true for the observed requests $W$ since the censoring functions depend on $P$ which in
turn depends on previous marks of $W$. Nevertheless, given $M_s = i$ for $0 \le s \le
\sigma_\ell$, the interarrival times of $W$ are i.i.d.\ $Exp(\lambda_i)$, and the distribution
of $Z_\ell$ is only a function of $P_{\sigma_{\ell-1}}$. Therefore, if we take a sample path of
$W$ where $r$-many arrivals are observed on $[0,t]$, then the likelihood of this path would be
written as $\P^{\vp,a,\xi} \{ \sigma_k \in dt_k , Z_k \in d{\vz}_k, \,  \sigma_r \leq t \,; \,
k\le r \,  | \,M_s =i , s\le t \} = $
\begin{align}
\label{path-likelihood-when-M-is-fixed}  [ \lambda_i\e^{-\lambda_i t_1 } dt_1] \cdots
[\lambda_i\e^{-\lambda_i (t_r - t_{r-1}) } dt_m]  \e^{-\lambda_i (t-t_r) } \prod_{k =1}^r
\Bigl[ \sum_y f_i (y) \mathcal{C}(y, \vz_k; P_{t_k-}) \Bigr] =\e^{-\lambda_i t } \prod_{k =1}^r
\lambda_i dt_k \cdot g_i (\vz_k ; P_{t_k-} ),
\end{align}
where $$ g_i(\vz; p) \triangleq \sum_{y \in \{1,\ldots, R\}} f_i(y) \mathcal{C}(y,\vz; p),$$
denotes the conditional likelihood of a request of type $\vz$ (which is just the sum of
conditional likelihood of all possible corresponding order sizes $y$). Note that in the case
with censoring $g_i( (z^1,z^2); p) = \sum_{n=\lfloor p \rfloor + 1}^R f_i(n) 1_{\{z^2=\Delta\}}
+ f_i(z^1) 1_{\{z^2 = 0\}}$.

More generally, we obtain
\begin{multline}
\label{path-likelihood-given-M}
1_{\{ M_t =i\}} \cdot \P^{\vp,a} \Big\{ \sigma_\ell \in dt_\ell , Z_\ell = \vz_\ell, \sigma_r \leq t; \, \ell\le r \, \Big| \, M_s, s \le t \Big\} \\
= 1_{\{ M_t =i\}} \cdot \exp{ \left( - \int_0^t \sum_{j=1}^n \lambda_j  1_{ \{ M_{s} =j \} } ds
\right) } \cdot  \prod_{\ell=1}^r \left( \sum_{j \in E}  1_{\{ M_{t_\ell} =j \}} [ \lambda_j
dt_\ell \cdot g_j (\vz_\ell ; P_{t_\ell-}) ] \right).
\end{multline}

The above observation leads to the description of the paths of the sufficient statistics $(\vP,P)$.

\begin{prop}\normalfont
\label{cor:pdp}
Let us define $\vx (t, \vp) \equiv (x_1(t, \vp), \ldots , x_m(t,
\vp))$ via
\begin{align}
\label{eq:x-i} x_i(t, \vp) \triangleq \frac{    \P^{\vp}  \{ \sigma_1 > t , M_t =i  \}  }
 {   \P^{\vp}  \{ \sigma_1 > t  \} }
=\frac{  \E^{\vp} \left[ 1_{\{ M_{t} =i\}} \cdot \e^{ - I(t)}   \right]  } {\E^{ \vp}
\left[   \e^{ - I(t)}   \right] }
 , \qquad \text{for $i \in E$}.
\end{align}

Then the paths of $(\vP,P)$ can be described by
     \begin{align}\label{eq:rel-pi-x} \left\{
     \begin{aligned}
    \vP(t)&=  \vx \left(t-\sigma_\ell,\vP({\sigma_\ell})\right),  \qquad
\quad \sigma_\ell \leq t< \sigma_{\ell+1}, \;\; \ell\in \mathbb{N} \quad \\
\Pi_i(\sigma_\ell)&= \frac{  \lambda_i f_i(Z^1_\ell) \Pi_i(\sigma_\ell-) }{ \sum_{j \in E}
\lambda_j f_j(Z^1_\ell) \Pi_j(\sigma_\ell-) } \qquad\qquad\text{if }\quad Z^2_\ell = 0;
\\ \Pi_i(\sigma_\ell)&= \frac{ \sum_{y=\lfloor P(\sigma_\ell-)\rfloor+1}^R \lambda_i f_i(y)
\Pi_i(\sigma_\ell-) }{ \sum_{y=\lfloor P(\sigma_\ell-)\rfloor+1}^R \sum_{j \in E} \lambda_j
f_j(y) \Pi_j(\sigma_\ell-) } \qquad\text{if }\quad Z^2_\ell = \Delta; \\
P(\sigma_\ell) & = P(\sigma_\ell-)-Z^1_\ell;\\
P(\tau_k) & = P(\tau_k-) + \xi_k.
\end{aligned} \right\}
\end{align}
\end{prop}

\begin{proof}
See Section~\ref{sec:pfpdp}. The main idea is to express $\vP_i(t)$ as a ratio of likelihood
functions, and to use \eqref{path-likelihood-given-M} to obtain explicit formulas for
likelihood of different observations conditional on the state of $M$.
\end{proof}

The deterministic paths described by $\vx$ come from a first-order ordinary differential
equation. To observe this fact first recall that the components of the vector
\begin{align}\label{def:m}
\vec{m} (t , \vp ) \equiv ( m_1 (t , \vp ), \ldots , m_m (t , \vp ) ) \triangleq \Bigl( \,
\E^{\vp,a} \left[ 1_{\{ M_{t} =1\}} \cdot \e^{ - I(t)}   \right]  , \ldots ,
  \E^{\vp,a} \left[ 1_{\{ M_{t} =m\}} \cdot \e^{ - I(t)}   \right] \, \Bigr)
\end{align}
solve $ d m_i (t, \vp ) / dt  = - \lambda_i m_i (t,\vp ) + \sum_{j \in E}  m_j (t,\vp) \cdot
q_{j, i}$ (see e.g. \cite{DarrochMorris,Neuts,KarlinTaylor2}). Now using (\ref{eq:x-i})  and
applying the chain rule we obtain
\begin{align}
\label{eq:vx-dyn}
\frac{d x_i (t,\vp) }{dt} = \left( \sum_{j\in E} q_{j,i}  x_j (t,\vp) - \lambda_i  x_i (t,\vp)  + x_i (t,\vp)   \sum_{j \in E} \lambda_{j}  x_j (t,\vp)  \right).
\end{align}

Note that since $(\vP,P)$ is a piecewise deterministic Markov process by the last proposition,
the results of \cite{davis93} imply that this pair is a strong Markov process.

\section{Characterization and Continuity of the Value Function}
\label{sec:sequential}

A standard approach (see e.g.\ \cite{BertsekasBookVol1}) to solving stochastic control problems
makes use of the dynamic programming (DP) principle. Heuristically, the DP implies that to
implement an optimal policy, the manager should continuously compare the intervention value,
i.e.\ the maximum value that can be extracted by immediately ordering the most beneficial
amount of inventory, with the continuation value, i.e.\ maximum value that can be extracted by
doing nothing for the time being. In continuous time this leads to a recursive equation that
couples $U(t,\vp,a)$ with $U(t-dt,\vp,b)$ for $b \ge a$. Such a coupled equation could then be
solved inductively starting with the known value of $U(0, \vp, a)$.

In this section we show that the above intuition is correct and that $U$ satisfies a coupled
optimal stopping problem. More precisely, we show in Theorem \ref{prop:mnprp} that it is the
unique fixed point of a functional operator $\G$ that maps functions to value functions of
optimal stopping problems. This gives a direct and self-contained proof of the DP for our
problem. We also show that the sequence of functions that we obtain by iterating $\G$ starting
at the value of no-action converges to the value function uniformly. Since $\G$ maps continuous
functions to continuous functions and the convergence is uniform, we obtain that the value
function $U$ is jointly continuous with respect to all of its variables. Continuity of the
value function leads to a direct characterization of an optimal strategy in Proposition
\ref{prop:opt-strat}. The analysis of this section parallels the general (infinite horizon)
framework of impulse control of piecewise deterministic processes (pdp) developed by
\cite{CostaDavis89}. We should also point out that Theorem~\ref{prop:mnprp} is used to
establish an alternative characterization of the value function, see Proposition~\ref{eq:hat-wn},
which is more amenable to computing the value function.

First, we will analyze the problem with no intervention. This analysis will facilitate the
proofs of the main results in the subsequent subsection.
\subsection{Analysis of the Problem with no Intervention}
Let $U_0$ be
the value of no-action, i.e.,
\begin{equation}\label{def:U-0}
\begin{split}
U_0(T,\vp,a)&=\E^{\vp,a} \left[   \int_0^{T} \e^{- \rho t }  \, c(P_t) \, dt + \sum_{k :
\sigma_k \le T} \e^{-\rho \sigma_k }K( (Y_k - P_{\sigma_k-})_+)\right]
\\&=\E^{\vp,a} \left[   \int_0^{T} \e^{- \rho t }  \, c(P_t) \, dt + \sum_{k :
\sigma_k \le T} \e^{-\rho \sigma_k }\sum_{i \in E}\Pi^{(i)}_{\sigma_k}\int_{\R_+}K( (y - P_{\sigma_k-})_+)\nu_i(dy)\right].
\end{split}
\end{equation}
We will prove the continuity of  $U_0$ in the next proposition; this property will become crucial in the proof of the main result of the next section. But before let us present an auxiliary lemma which will help us prove this proposition.

\begin{lemm}\label{lem:fst-est}
For all $n \geq 2$, we have the uniform bound
\begin{equation}
 \P^{\vp,a}\{T>\sigma_n\} \leq \frac{\bar{\lambda}T}{n-1}.
\end{equation}
\end{lemm}

\begin{proof}

\textbf{Step 1}.First we will show that
\begin{equation}\label{eq:laplacetranssigma}
\E^{\vp,a}\left[\e^{-u \sigma_n}\right] \leq
\left(\frac{\bar{\lambda}}{\bar{\lambda}+u}\right)^n.
\end{equation}
The conditional probability of the first jump satisfies $\P^{\vp}\{\sigma_1>t|M\}=\e^{-I(t)}$.
Therefore,
\begin{equation}\label{eq:fst-est}
\begin{split}
\E^{\vp,a}\left[\e^{-u \sigma_1}|M\right]=\E^{\vp,a}\left[\int_{\sigma_1}^{\infty}u\e^{-u t}\,dt
\Big| \, M\right] &=\int_0^{\infty}\P^{\vp}\{\sigma_1 \leq t | M \} u\e^{-ut}\,dt
\\& = \int_0^{\infty}\left[1-\e^{-I(t)}\right] u\e^{-ut}\, dt \\
&\leq \int_0^{\infty}\left[1-\e^{-\bar{\lambda} t}\right]u\e^{-u t}dt =
\frac{\bar{\lambda}}{u+\bar{\lambda}}.
\end{split}
\end{equation}
Since the observed process
$X$ has independent increments given $M$, it readily follows that $\E^{\vp,a}\left[\e^{-u
\sigma_n}|M\right] \leq \bar{\lambda}^n/(\bar{\lambda}+u)^n$, which immediately implies (\ref{eq:laplacetranssigma}).

\textbf{Step 2.} Note that
\begin{equation}
 \P^{\vp,a}\{T>\sigma_n\} \leq  \E^{\vp}\left[1_{\{T>\sigma_n\}}(T/\sigma_n)\right] \leq \E^{\vp}\left[T/\sigma_n\right].
\end{equation}
Since $1/\sigma_n=\int_0^{\infty}\e^{-\sigma_n u} du$, an application of Fubini's theorem
together with \eqref{eq:laplacetranssigma}
\begin{equation}
\E^{\vp} \left[\frac{1}{\sigma_n}\right] \leq \int_0^{\infty}\left(\frac{\bar{\lambda}}{\bar{\lambda}+u}\right)^n du =\frac{\bar{\lambda}}{n-1}, \quad n \geq 2,
\end{equation}
implies the result.
\end{proof}

Define the jump operators $S_i$ through their action on a test function $w$ by
\begin{multline}
\label{def:S} S_i w(t, \vp, a) \triangleq \int_{\R} \Big\{ w \left(t, \left(\, \frac{ \lambda_1
f_1(y) \pi_1 }{ \sum_{j \in E} \lambda_j f_j(y) \pi_j }, \ldots, \frac{  \lambda_m
f_m(y) \pi_m }{ \sum_{j \in E} \lambda_j f_j(y) \pi_j } \right), (a-y)_+\right) \\
 + K( (y-a)_+) \Bigr\} \nu_i(dy), \quad \text{for $i \in E$.}
\end{multline}
The motivation for $S_i$ comes from the dynamics of $\vP$ in Proposition \ref{cor:pdp} and
studying expected costs if an immediate demand order (of size $y$) arrives.

\begin{prop}\label{prop:contUo}
$U_0$ is a continuous function.
\end{prop}

\begin{proof}
Let us define a functional operator $\Upsilon$ through its action on a test function $w$ by
\begin{equation}\label{eq:defn-I}
\begin{split}
\Upsilon w(T,\vp,a)&= \E^{\vp,a}\bigg[\int_0^{\sigma_1\wedge T}\e^{-\rho t}c(p(t,a))dt +1_{\{\sigma_1
\leq T\}}\bigg( \e^{-\rho \sigma_1 }\sum_{i \in E}\Pi^{(i)}_{\sigma_1}\int_{\R_+}K( (y - P_{\sigma_1-})_+)\nu_i(dy)
\\& \qquad +w(T-\sigma_1,\vP_{\sigma_1},P_{\sigma_1})\bigg)\bigg]
\\&= \int_0^{T}\e^{-\rho u} \sum_{i \in E} m_{i}(u,\vp) \cdot \left[c(p(u,a))+\lambda_i \cdot S_i w(T-u, \vec{x}(u,\vp),
p(u,a))\right]du.
\end{split}
\end{equation}
The operator $\Upsilon$ is motivated by studying expected costs up to and including the first
demand order assuming no-action on the part of the manager.

It is clear from the last line of \eqref{eq:defn-I} that $\Upsilon$ maps continuous functions
to continuous functions. As a result of the strong Markov property of $(\vP,P)$ we observe that
$U_0$ is a fixed point of $\Upsilon$, and if we define
\begin{equation}\label{def:k-n}
k_{n+1}(T,\vp,a)= \Upsilon k_n(T,\vp,a), \quad k_0(T,\vp,a)=0, \quad T \in \R_+, \vp \in D, a \in \A
\end{equation}
then $k_n \nearrow U_0$ (also see Proposition 1 in \cite{CostaDavis89}).  To complete our proof
we will show that $k_n$ converges to $U_0$ uniformly (locally in the $T$ variable); since all the elements of the sequence
$(k_n)_{n \in \N}$ are continuous the result then follows.

Using the strong Markov property we can write $k_n$ as
\begin{equation}\label{eq:alt-rep-kn}
k_n(T,\vp,a)= \E^{\vp,a}\left[\int_0^{\sigma_n\wedge T}\e^{-\rho t}c(P_t)\,dt+\sum_{k :
\sigma_k \le T \wedge \sigma_n} \e^{-\rho \sigma_k }K( (Y_k - P_{\sigma_k-})_+)\right],
\end{equation}
from which it follows that \begin{equation}\label{eq:kn-U0}
\begin{split}
|U_0(T,\vp,a)&-k_n(T,\vp,a)| \leq
\E^{\vp,a}\left[1_{\{T>\sigma_n\}}\left(\int_{\sigma_n}^{T}\e^{-\rho t} c(\bar{P})dt+ K(R)
\sum_{k \geq n}\e^{-\rho \sigma_k}\right)\right]
\\& \leq \E^{\vp,a}\left[1_{\{T>\sigma_n\}}\left(c(\bar{P})T+K(R) (N(T)-N(\sigma_n)) \right)\right]
\\ &\leq  \P^{\vp,a}\{T>\sigma_n\} c(\bar{P})T+ K(R)\E^{\vp,a}\left[1_{\{T>\sigma_n\}} N(T) \right]
\\ &\leq   \P^{\vp,a}\{T>\sigma_n\} c(\bar{P})T+K(R) \sqrt{ \P^{\vp,a}\{T>\sigma_n\} } \sqrt{\E^{\vp,a}[N(T)^2]},
\end{split}
\end{equation}
where we used the Cauchy-Schwarz inequality to obtain the last inequality. Using
Lemma~\ref{lem:fst-est}, and $\E^{\vp,a}[N(T)^2] \leq \bar{\lambda} T+(\bar{\lambda} T)^2$ we
obtain
\begin{equation}\label{eq:unconkn}
\begin{split}
|U_0(T,\vp,a)-k_n(T,\vp,a)| &\leq c(\bar{P})T
\frac{\bar{\lambda}}{n-1}+K(R)\sqrt{\bar{\lambda}T+(\bar{\lambda}
T)^2}\sqrt{\frac{\bar{\lambda}T}{n-1}}
\\&\leq c(\bar{P})\bar{T} \frac{\bar{\lambda}}{n-1}+K(R)\sqrt{\bar{\lambda}\,\,\bar{T}+\bar{\lambda}^2 \bar{T}^2}\sqrt{\frac{\bar{\lambda}T}{n-1}},
\end{split}
\end{equation}
for any $T \in [0,\bar{T}]$. Letting $n \rightarrow \infty$ we see that $(k_n)_{n \in
\mathbb{N}}$ converges to $U_0$ uniformly on $[0,\bar{T}]$. Since $\bar{T}$ is arbitrary, the
result follows.
\end{proof}

\subsection{Dynamic Programming Principle and an Optimal Control}\label{sec:optimal-stragegy}
We are now in position to establish the DP for $U$ and also characterize an optimal strategy.
In our problem the DP takes the form of a coupled optimal stopping problem of Theorem
\ref{prop:mnprp}.

Let us introduce a functional operator $\M$ whose action on a test function $w$ is
\begin{align}\label{eq:intervention-op} \M w(T,\vp,a) \triangleq \min_{b : a \leq  b \leq  \bar{P}}
 \Bigl\{w(T, \vp, b) + h(b-a)+\zeta \Bigr\}.
\end{align}
The operator $\M$ is called the \emph{intervention} operator and denotes the minimum cost that
can be achieved if an immediate supply of size $b-a$ is made.
\begin{lemm}\label{lem:mcont}
The operator $\mathcal{M}$ maps continuous functions to continuous functions.
\end{lemm}
\begin{proof}
This result follows since the set valued map $a \rightarrow [a,\bar{P}]$ is continuous (see Proposition D.3 of \cite{MR995463}).
\end{proof}

We will denote the smallest supply
order the minimum in \eqref{eq:intervention-op} by
\begin{align}\label{eq:optimal-action}
d_{\M w}(T,\vp, a) \triangleq \min \Bigl\{ b \in [a, \bar{P}]:\, w(T,\vp, b)
+ h(b-a) +\zeta= \M w(T,\vp,a) \Bigr\}.
\end{align}

Let us define a functional operator $\G$ by its action on a test function $V$ as
\begin{align}\label{def:G}
\G V(T, \vp,a) & =\inf_{\tau \in \s(T)} \E^{\vp,a} \Bigl[ \int_0^\tau \e^{-\rho s} c( P_s) \, ds + \sum_\ell \sum_{i}\e^{- \rho \sigma_\ell}1_{\{ \sigma_\ell \leq \tau \}}
\vP^{(i)}_{\sigma_l} \int_{\R_+}K( (y- P_{\sigma_\ell}-)_+)\nu_i(dy)  \\  \notag & \qquad + \e^{-\rho \tau} \M V(T-\tau,
\vP_{\tau}, P_{\tau}) \Bigr],
\end{align}
for $\,T \in \R_+,\, \vp \in D$, and $a \in \A$. The above definition is motivated by studying
minimal expected costs incurred by the manager until the first supply order time $\tau$.

\begin{lemm}\label{lemma:gmapsconttocont}
The operator $\G$ maps continuous functions to continuous functions.
\end{lemm}

\begin{proof}
This follows as a result of Lemma~\ref{lem:mcont}: as shown in Corollary 3.1 of \cite{LS07}
(see also Remark 3.4 in \cite{BS06}), when $\M w$ is continuous, then the value function $\G w$
of this optimal stopping problem is also continuous.
\end{proof}

Let $V_0 \triangleq U_0$ (from \eqref{def:U-0}) and
\begin{equation}\label{defn:Un}
V_{n+1} \triangleq \G V_n, \quad n \geq 0.
\end{equation}
Clearly, since $\G$ is a monotone/positive operator, i.e. for any two functions $f_1 \leq f_2$
we have  $\G f_1 \leq \G f_2$, and since $V_1 \leq V_0$, $(V_n)_{n \in \mathbb{N}}$ is a
decreasing sequence of functions. The next two propositions show that this sequence converges
(point-wise) to the value function, and that the value function satisfies the dynamic
programming principle. Similar results were presented in Propositions 3.2 and 3.3 in
\cite{baylud08} (for a problem in which the controls do not interact with the observations).
The proofs of the following propositions are similar, and hence we give them in the Appendix
for the reader's convenience.

\begin{lemm}\label{prop:vnconvU}
$V_{n}(T,\vp,a) \downarrow U(T,\vp,a), \quad\text{for any}\quad T \in \R_+$, $\vp \in D$, $a
\in \A$.
\end{lemm}

\begin{proof}
The proof makes use of the fact that the value functions defined by restricting the admissible
strategies  to the ones with at most $n \ge 1$ supply orders up to time $T$ can be obtained by
iterating operator $\G$ $n$-times (starting from $U_0$). This preliminary result is developed
in Appendix~\ref{sec:analysis1}. The details of the proof can be found in
Appendix~\ref{sec:proofofconvprop}.
\end{proof}

\begin{prop}\label{prop:dp}
The value function $U$ is the largest solution of the dynamic programming equation $\G U=U$,
such that $U \leq U_0$.
\end{prop}
\begin{proof}
The result follows from the monotonicity of $\G$ and Proposition~\ref{prop:vnconvU}.
See Appendix~\ref{sec:proofofprop31} for the details.
\end{proof}
The Theorem below improves the results of Proposition~\ref{prop:dp} and helps us describe an optimal policy. Let us first point out that $U_0$ and hence $U$ are bounded.

\begin{rem}\label{rem:bddu}
It can be observed from the proof of Proposition~\ref{prop:mnprp} that the value function
$U(T,\cdot,a)$ is uniformly bounded,
\begin{align*}
 0 \leq U(T,\vp,a) & \le U_0(T,\vp,a) = \E^{\vp,a} \left[ \int_0^T \e^{-\rho s} c(P_s) \,ds
+ \sum_{k} \e^{-\rho \sigma_k}1_{\{ \sigma_k \le T\}}K( (Y_k - P_{\sigma_k-})_+) \right] \\
 & \le \int_0^T \e^{-\rho s} c(\bar{P}) \, ds + \E^{\vp,a} \left[ \sum_k \e^{-\rho \sigma_k}1_{\{ \sigma_k \le
 T\}}K(R) \right] \\
 & \le c(\bar{P}) T + K(R) \E^{\vp,a}[ N(T)] \leq [c(\bar{P}) +K(R) \bar{\lambda}] \cdot T,
 \end{align*}
 since $N$ is a counting process with maximum intensity $\bar{\lambda}$.
\end{rem}

Below is the main result of this section.

\begin{thm}\label{prop:mnprp}
The value function $U$ is the unique fixed point of $\G$ and it is continuous.
\end{thm}

\begin{proof}
\textbf{Step 1.}
 Let us fix $\bar{T}>0$.
 We will first show that  $U$ is the unique fixed point of $\G$ and that $(V_n)_{n \in \mathbb{N}}$ converges to $U$ uniformly on $T \in [0, \bar{T}]$, $\vp \in D$, $a \in \A$.
 Let us restrict our functions $U_0$ and $U$ to
$T \in [0, \bar{T}]$, $\vp \in D$, $a \in \A$. And we will consider the restriction of $\G$ that acts on functions that are defined on $T \in [0, \bar{T}]$, $\vp \in D$, $a \in \A$.
Thanks to Lemma 1 of \cite{MR1150206} (also see \cite{zab83}) it is enough to show that $U_0 \leq k \G 0$ for some $k(\bar{T})>0$ (We showed that $U_0$ is continuous in Proposition~\ref{prop:contUo} and that $U_0$ is bounded in $[0,\bar{T}]$ in Remark~\ref{rem:bddu}
 in order to apply this lemma).
For any stopping time $\tau \leq T$
\begin{equation}\label{eq:U0-vs-G0}
\begin{split}
U_0(T,\vp,a)&=\E^{\vp,a} \bigg[   \int_0^{\tau} \e^{- \rho t }  \, c(P_t) \, dt   + \sum_{k :
\sigma_k \le \tau} \e^{-\rho \sigma_k }K( (Y_k - P_{\sigma_k-})_+)
\\& \qquad +\int_{\tau}^{T}  \e^{- \rho t }  \, c(P_t) \, dt  +\sum_{k :
\sigma_k \in (\tau,T]} \e^{-\rho \sigma_k }K( (Y_k - P_{\sigma_k-})_+) \bigg].
\end{split}
\end{equation}
Next, we will provide upper bounds for the terms in the second line of \eqref{eq:U0-vs-G0}.
First, note that
\begin{equation}\label{eq:est-1}
\int_{0}^{T} \e^{-\rho t}c(P_t) \, dt \leq C_\rho(\tau) \triangleq \left\{
\begin{aligned} \frac{c(\bar{P})}{\rho}\e^{-\rho \tau} & \quad\text{ if }\quad \rho > 0 \\
(T-\tau)c(\bar{P}) & \quad\text{ if }\quad \rho=0.\end{aligned}\right.
\end{equation}
Second,
\begin{equation}\label{eq:sumkK}
\sum_{k : \sigma_k \in (\tau,T]} \e^{-\rho \sigma_k }K( (Y_k - P_{\sigma_k-})_+)  \leq\e^{-\rho
\tau} K(R)\sum_k 1_{\{\sigma_k \in (\tau,T]\}}.
\end{equation}
The expected value of sum on the right-hand-side of (\ref{eq:sumkK}) is bounded above by a
constant, namely
\begin{equation}\label{eq:est-lst}
\E^{\vp,a} \left[ \sum_k 1_{\{\sigma_k \in (\tau,T]\}}\right] = \E^{\vp,a}[ N(T) - N(\tau)] \le
\bar{\lambda}T.
%\leq\e^{\rho T} \E^{\vp,a} \left[ \sum_k
%\e^{-\rho \sigma_k}1_{\{ \sigma_k \le T\}} \right] \leq\e^{\rho T}\sum_k
%\left(\frac{\bar{\lambda}}{\bar{\lambda}+\rho}\right)^k <\e^{\rho T}
%\left(1+\frac{\bar{\lambda}}{\rho}\right),
\end{equation}

%where the econd inequality follows from Lemma~\ref{lem:fst-est}.
Using the estimates developed in \eqref{eq:est-1}-\eqref{eq:est-lst} back  in
\eqref{eq:U0-vs-G0}, we obtain that
\begin{equation}\label{eq:bndU0-G01}
\begin{split}
U_0(T,\vp,a) & \leq \E^{\vp,a} \bigg[   \int_0^{\tau} \e^{- \rho t }  \, c(P_t) \, dt   +
\sum_{k : \sigma_k \le \tau} \e^{-\rho \sigma_k }K( (Y_k - P_{\sigma_k-})_+) +\e^{-\rho \tau}
(C_\rho(\tau) + K(R)\bar{\lambda}T)\bigg]
\\ & \leq \left(\frac{C_\rho(0) + K(R)\bar{\lambda}T}{\zeta} \vee 1\right) \E^{\vp,a} \bigg[   \int_0^{\tau} \e^{- \rho t }  \, c(P_t) \, dt   + \sum_{k :
\sigma_k \le \tau} \e^{-\rho \sigma_k }K( (Y_k - P_{\sigma_k-})_+) +\e^{-\rho \tau}\zeta\bigg]
\end{split}
\end{equation}
Minimizing the right-hand-side over all admissible stopping times $\tau$ we obtain that $U_0
(T,\vp,a)$
\begin{equation*}
\begin{split}
&\leq \left(\frac{C_\rho(0) + K(R)\bar{\lambda}T}{\zeta} \vee 1\right) \inf_{\tau \in
\S(T)}\E^{\vp,a} \bigg[   \int_0^{\tau} \e^{- \rho t }  \, c(P_t) \, dt   + \sum_{k : \sigma_k
\le \tau} \e^{-\rho \sigma_k }K( (Y_k - P_{\sigma_k-})_+) +\e^{-\rho \tau}\zeta\bigg]
\\ &\leq \left(\frac{C_\rho(0) + K(R)\bar{\lambda}T}{\zeta} \vee 1\right) \G 0 \leq \left(\frac{C_\rho + K(R)\bar{\lambda}\,\,\bar{T}}{\zeta} \vee 1\right) \G 0,
\end{split}
\end{equation*}
which establishes the desired result. Moreover since $\bar{T}$ is arbitrary we see that $U$ is indeed the unique fixed point of $\G$ among all the functions defined on $T \in \R_+$, $\vp
\in D$, $a \in \A$.

\noindent \textbf{Step 2.} We will show that $U$ is continuous. Since $(V_n)_{n \in \mathbb{N}}$  converges to $U$ uniformly on $T \in [0, \bar{T}]$, $\vp
\in D$, $a \in \A$ for any $\bar{T}<\infty$ the proof will follows once we can show that
every element in the sequence $(V_n)_{n \in \mathbb{N}}$ is continuous. But this result follows from Lemma~\ref{lemma:gmapsconttocont} and the continuity of $U_0$.
\end{proof}

Using the continuity of the value function, one can prove that the strategy given in the next
proposition is optimal. The proof is analogous to the proof of Proposition 4.1 of
\cite{baylud08}.
\begin{prop}\label{prop:opt-strat}
Let us iteratively define  $\xi^* = (\xi_0, \tau_0; \xi_1, \tau_1, \ldots)$ via
$\xi_0 = a, \tau_0 = 0$ and
\begin{align}
\left\{
\begin{aligned}
\tau_{k+1} &= \inf \left\{s \in [\tau_k,T] \colon U(T-s,\vP(s),P_s) = \M U(T-s,\vP(s), P_s) \right\}; \\
\xi_{k+1} & = d_{\M U}(T-\tau_{k+1},\vP({\tau_{k+1}}),P_{\tau_{k+1}}), \qquad k = 0,1,\ldots,
\end{aligned} \right.
\end{align}
with the convention that $\inf \emptyset =T+\eps$, $\eps>0$, and $\tau_{k+1}=0$.
Then $\xi^*$ is an optimal strategy for \eqref{def:U}.
\end{prop}
Proposition \ref{prop:mnprp} implies that to implement an optimal policy the manager should
continuously compare the intervention value $\M U$ versus the value function $U \ge \M U$. As
long as, $U > \M U$, it is optimal to do nothing; as soon as $U = \M U$, new inventory in the
amount $d_{\M U}$ should be ordered. The overall structure thus corresponds to a time- and
belief-dependent $(s,S)$ strategy which matches the intuition of real-life inventory managers.

\begin{rem}\label{rem:qvi}
As a result of the dynamic programming principle, proved in Theorem~\ref{prop:mnprp}, the value
function $U$ is also expected to be the unique weak solution of  a coupled system of QVIs
(quasi-variational inequalities)
\begin{equation}
\begin{split}
-\frac{\partial}{\partial T}U (T,\vp,a) + \mathcal{A} U  (T,\vp,a) -\rho  U  (T,\vp,a) + c(\vp,a) &\leq 0, \\
U  (T,\vp,a) &\geq \mathcal{M} U  (T,\vp,a),\\
\left(-\frac{\partial}{\partial T}U  (T,\vp,a) + \mathcal{A} U  (T,\vp,a) -\rho U  (T,\vp,a) +
c (\vp,a)\right)(U ( (T,\vp,a))-\mathcal{M}&U  (T,\vp,a))=0.
\end{split}
\end{equation}
Here $\mathcal{A}$ is the infinitesimal generator (first order integro-differential operator)
of the piece-wise deterministic Markov process $(\vP,P)$, whose paths are given by
Proposition~\ref{cor:pdp}. To determine $U$ one could attempt to numerically solve the above
multi-dimensional QVI. However, this is a non-trivial task. We will see that the value function
can be characterized in a way that naturally leads to a numerical implementation in
Section~\ref{sec:examples}. Also, having a weak solution is not good enough for existence of
optimal control, whereas in Theorem~\ref{prop:mnprp} we directly established the regularity
properties of $U$ which lead to a characterization of an optimal control.
\end{rem}

\subsection{Computation of the Value Function}\label{sec:examples} The characterization of the
value function $U$ as a fixed point of the operator $\G$ is not very amenable for actually
computing $U$. Indeed, solving the resulting coupled optimal stopping problems is generally a
major challenge. Recall that $U$ is also needed to obtain an optimal policy of Proposition
\ref{prop:opt-strat} which is the main item of interest for a practitioner.

To address these issues, in the next subsection we develop another dynamic programming equation
that is more suitable for numerical implementation. Namely, Proposition \ref{eq:hat-wn}
provides a representation for $U$ that involves only the operator $L$ which consists of a
deterministic optimization over time. This operator can then be easily approximated on a
computer using a time- and belief-space discretization. We have implemented such an algorithm
and in Section~\ref{sec:illust} then use this representation to give two numerical
illustrations.

%\subsection{A Computationally Feasible Representation of $U$}\label{sec:fjop}
We will show that the value function $U$ satisfies a second
dynamic programming principle, namely $U$ is the fixed point of the first jump operator
$L$, whose action on test function
$V$ is given by
\begin{multline}\label{eq:def-L}
L (V) (T,\vp,a) \triangleq \inf_{t \in [0,T]} \E^{\vp,a}\bigg[\int_0^{t \wedge
\sigma_1}\e^{-\rho s}c(p(s,a))\, ds \\ +1_{\{t <\sigma_1\}}\e^{-\rho t}\M V(T-t,\vP_t,P_t)
+\e^{-\rho \sigma_1}1_{\{t \geq \sigma_1\}} V(T-\sigma_1,\vP_{\sigma_1},P_{\sigma_1})\bigg].
\end{multline}
 This representation will be used in our numerical computations in Section
\ref{sec:illust}.
Observe that the operator $L$ is  monotone. Using the characterization of the stopping times of piecewise
deterministic Markov processes (Theorem T.33 \cite{bremaud}, and Theorem A2.3 \cite{davis93}),
which state that for any $\tau \in \S(T)$, $\tau \wedge \sigma_1 = t\wedge \sigma_1$ for some
constant $t$, we can write
\begin{equation}\label{eq:L-optimal-stopping}
\begin{split}
L (V) (T,\vp,a)&= \inf_{\tau \in \mathcal{S}(T)} \E^{\vp,a}\bigg[\int_0^{\tau \wedge
\sigma_1}\!\e^{-\rho s}c(p(s,a)) \,ds+1_{\{\tau<\sigma_1\}}\e^{-\rho \tau}\M
V(T-\tau,\vP_{\tau},P_{\tau})
\\&+\e^{-\rho \sigma_1}1_{\{\tau \geq \sigma_1\}} V(T-\sigma_1,\vP_{\sigma_1},P_{\sigma_1})\bigg],
\end{split}
\end{equation}

The following proposition gives the
characterization of $U$ that we will use in Section~\ref{sec:illust}. The proof of this result is carried out along the same lines as the proof of Proposition  3.4 of
\cite{baylud08}. The main ingredient is Theorem~\ref{prop:mnprp}.  We will skip the proof of
this result and leave it to the reader as an exercise.
\begin{prop}\label{eq:hat-wn} $U$ is the unique fixed point of $L$. Moreover, the
following sequence which is constructed by iterating $L$,
\begin{equation}\label{def:W}
W_0 \triangleq U_0, \quad W_{n+1} \triangleq L W_n, \quad n \in \mathbb{N},
\end{equation}
satisfies $W_n \searrow U$ (uniformly).
\end{prop}

\begin{rem}
Using Fubini's theorem, \eqref{eq:jumps-of-vP} and \eqref{def:m} we can write
$L$ as
\begin{multline}
\label{eq:J-expectations} L V (T,\vp,a)=  \inf_{0\leq t \leq T} \bigg\{\bigg(\sum_{i \in
E} m_i(t,\vp)  \bigg)
 \cdot \e^{- \rho t } \M
V\left(T-t,\vx( t, \vp ), p(t,a) \right) \\
+  \int_{0}^{t} \e^{- \rho u}     \sum_{i \in E} m_i(u,\vp) \cdot \Bigl(  c(p(u,a)) + \lambda_i
\cdot S_i V(T-u, \vx(u, \vp), p(u,a)) \Bigr) du\bigg\},
\end{multline}
in which $S_i$ is given by \eqref{def:S}. Observe that given future values of
$V(s,\cdot,\cdot)$, $s\le T$, finding $L(V)(T,\vp,a)$ involves just a deterministic
optimization over $t$'s.
\end{rem}

In our numerical computations below we discretize the interval $[0,T]$ and find the
deterministic supremum over $t$'s in \eqref{eq:J-expectations}. We also discretize the domain
$D$ using a rectangular multi-dimensional grid and use linear interpolation to evaluate the
jump operator $S_i$ of \eqref{def:S}. Because the algorithm proceeds forward in time with
$t=0,\Delta t, \ldots, T$, for a given time-step $t= m\Delta t$, the right-hand-side in
\eqref{eq:J-expectations} is known and we obtain $U(m\Delta t, \vp, a)$ directly. The
sequential approximation in \eqref{eq:hat-wn} is on the other hand useful for numerically
implementing infinite horizon problems.

\section{Numerical Illustrations}\label{sec:illust}
We now present two numerical examples that highlight the structure and various features of our
model. These examples were obtained by implementing the algorithm described in the last
paragraph of Section \ref{sec:examples}.

\subsection{Basic Illustration}\label{sec:ts-example} Our first example is based on the
computational analysis in \cite{TreharneSox}. The model in that paper is stated in
discrete-time; here we present a continuous-time analogue. Assume that the world  can be in
three possible states, $M_t \in E = \{ High, Medium, Low \}$. The corresponding demand
distributions are truncated negative binomial with maximum size $R=18$ and
$$ \nu_1 = NegBin(100, 0.99), \qquad \nu_2 = NegBin(900, 0.99), \qquad \nu_3 = NegBin(1600,
0.99).$$ This means that the expected demand sizes/standard deviations are $(1,1), (9,3)$ and
$(16,4)$ respectively.

The generator of $M$ is taken to be
$$ Q = \begin{pmatrix} -0.8 & 0.4 & 0.4 \\
0.4 & -0.8 & 0.4 \\
0.4 & 0.4 & -0.8 \end{pmatrix},$$ so that $M$ moves symmetrically and chaotically between its
three states. The horizon is $T=5$ with no discounting. Finally, the costs are
$$ c(a) = a, \qquad K(a) = 2a, \qquad h = 0, \qquad \zeta = 0,$$
so that there are zero procurement/ordering costs and linear storage/stockout costs. With zero
ordering costs, the controller must consider the trade-off between \emph{under}stocking and
\emph{over}stocking. Since excess inventory cannot be disposed (and there are no final salvage
costs), overstocking leads to higher future storage costs; these are increasing in the horizon
as the demand may be low and the stock will be carried forward for a long period of time. On
the other hand, understocking is penalized by the stock-out penalty $K$. The probability of the
stock-out is highly sensitive to the demand distribution, so that the cost of understocking is
intricately tied to current belief $\vP$. Summarizing, as the horizon increases, the optimal
level of stock decreases, as the relative cost of overstocking grows. Thus, as the horizon
approaches the controller stocks up (since that is free to do) in order to minimize possible
stock-outs. Overall, we obtain a time- and belief-dependent basestock policy as in
\cite{TreharneSox}.

\begin{figure}%[ht]\hspace{-1.5in}
\begin{tabular*}{\textwidth}{lr}
\begin{minipage}{3in}
\centering{\includegraphics[height=2.4in,width=3in,clip]{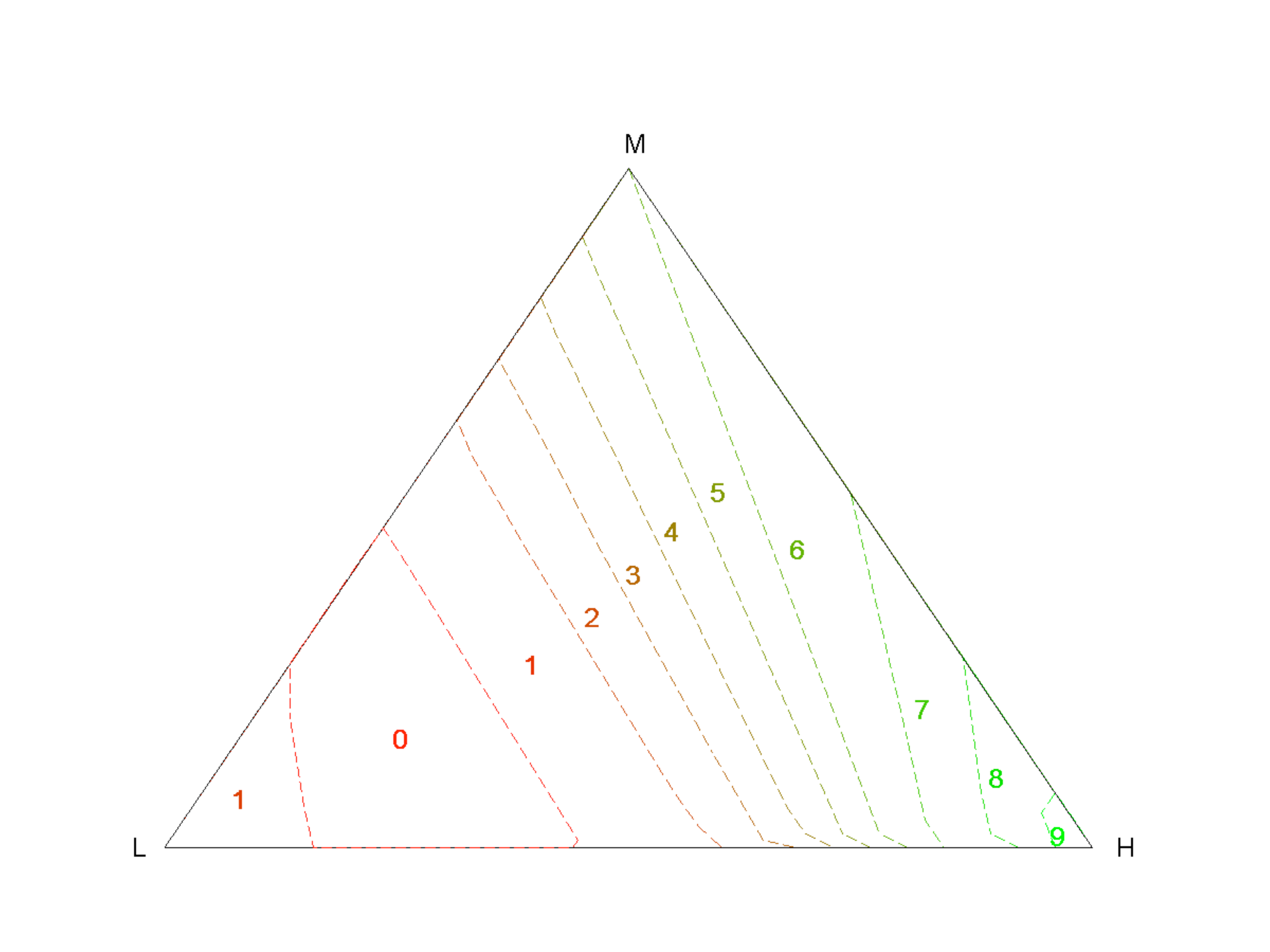}}

\end{minipage} &
\begin{minipage}{3in}
\centering{\includegraphics[height=2.4in,width=3in,clip]{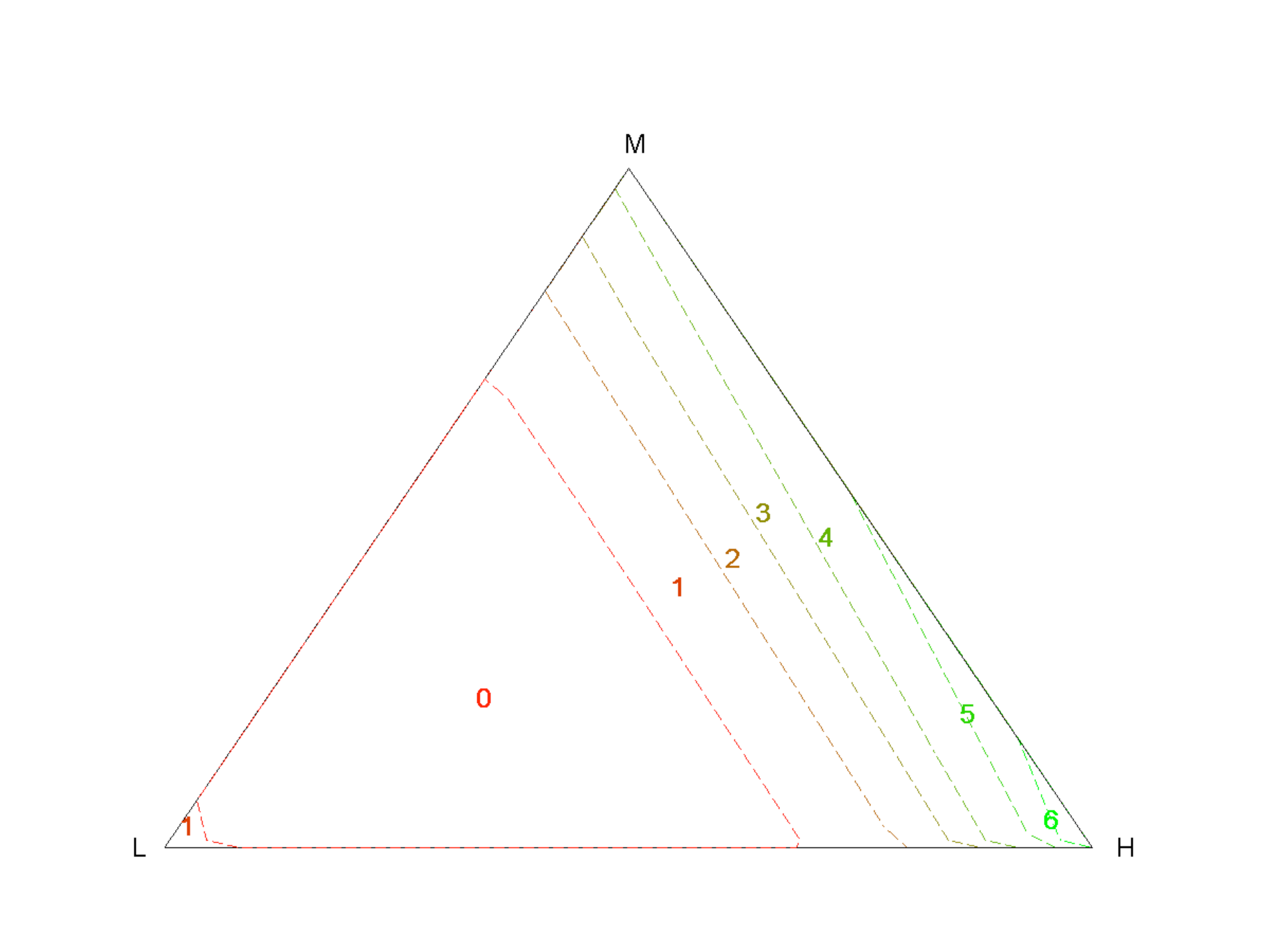}}

\end{minipage} \\
\begin{minipage}{3in}
\centering{\includegraphics[height=2.4in,width=3in,clip]{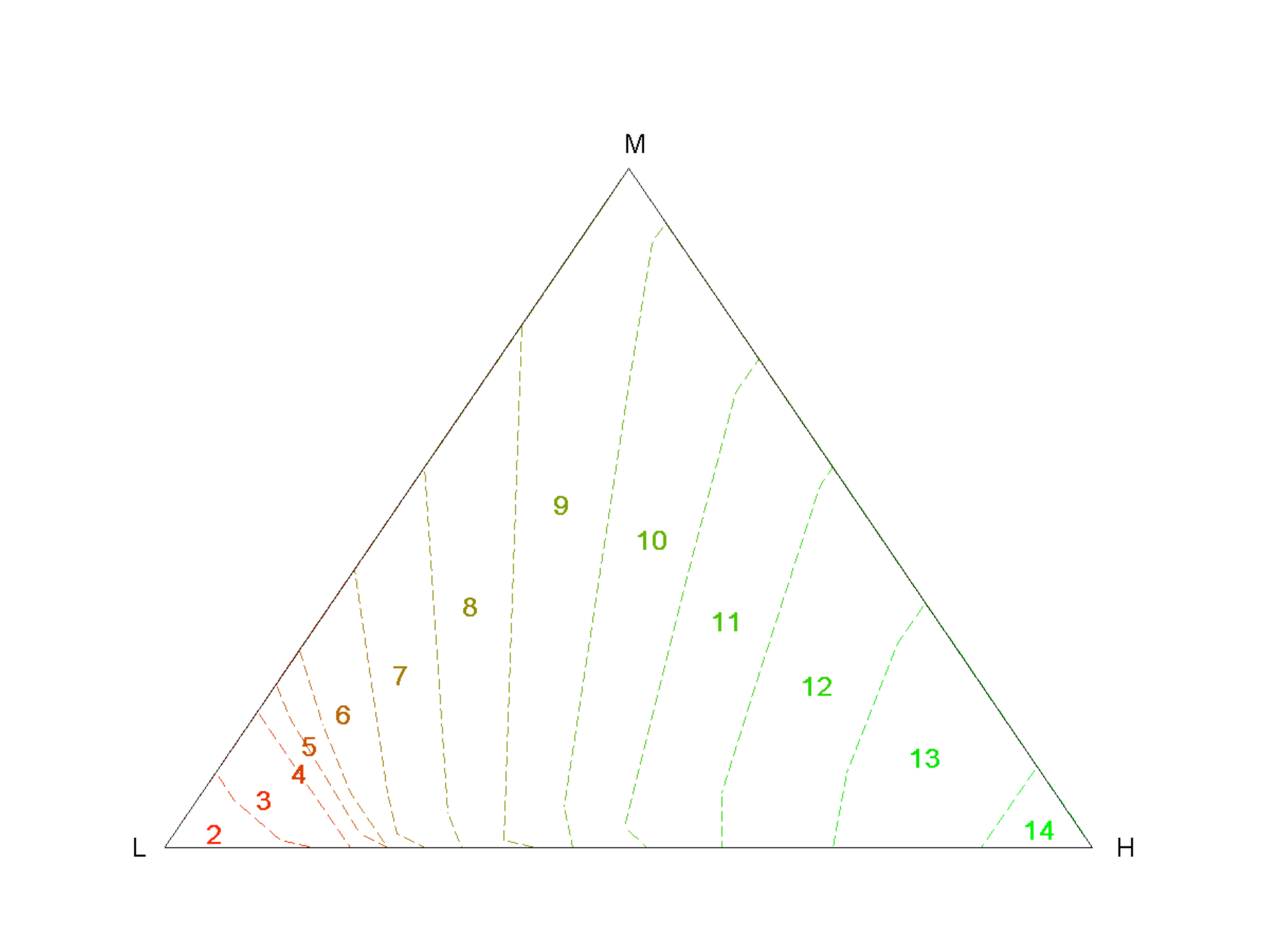}}

\end{minipage} &
\begin{minipage}{3in}
\centering{\includegraphics[height=2.4in,width=3in,clip]{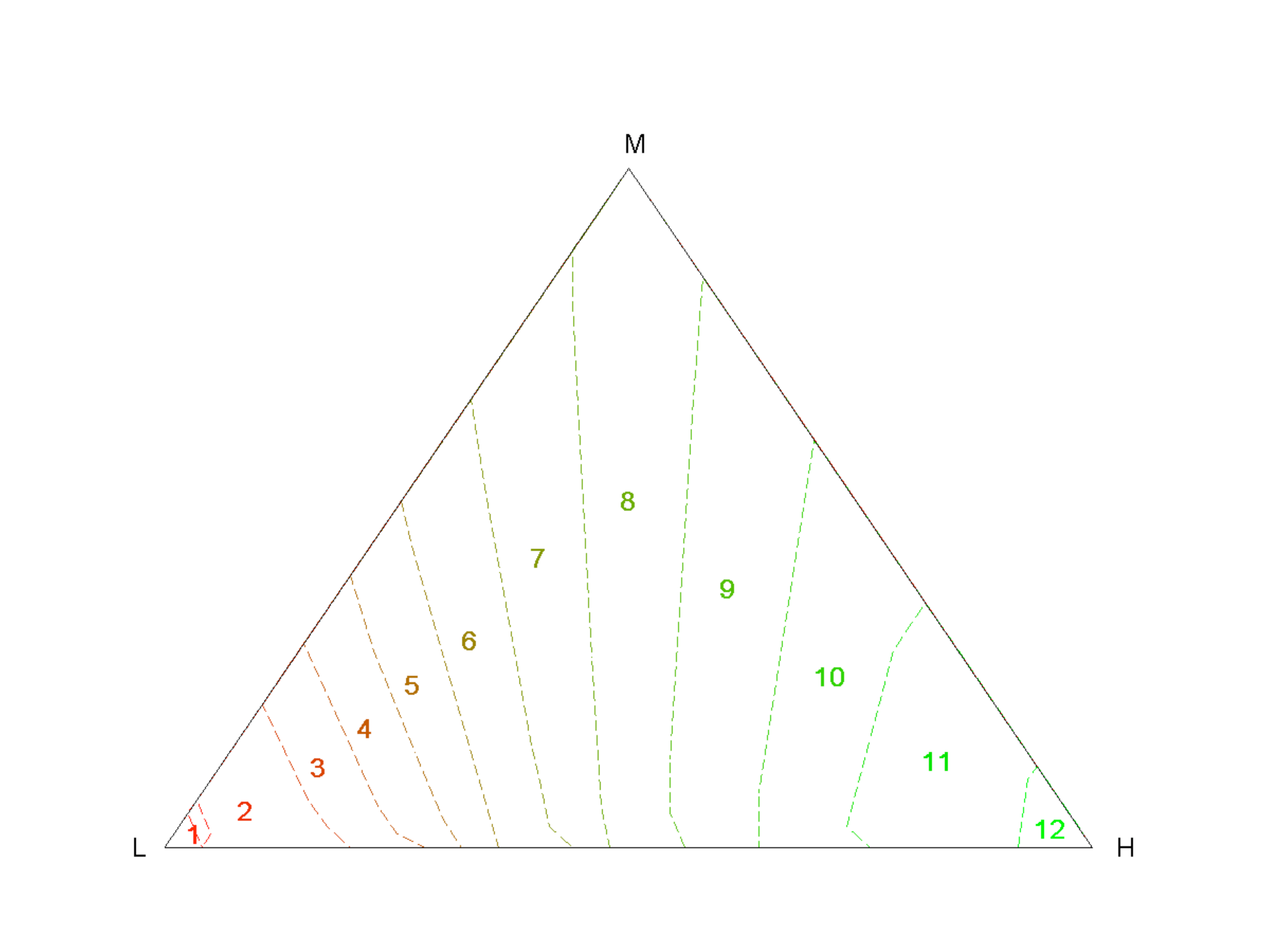}}

\end{minipage}
\end{tabular*}
\caption{Optimal inventory levels for different time horizons for Section \ref{sec:ts-example}.
We plot the regions of constancy for $d^*(T,\vp) = \argmax_a U(T,\vp, a)$, which is
the optimal inventory level to maintain given beliefs $\vp$ (since ordering costs are zero),
$\vp \in D=\{\pi_1+\pi_2+\pi_3 = 1\}$. Top panels: $K(a) = 2a$, bottom panels $K(a)=4a$; left
panels: $T=1$, right panels: $T=5$. \label{fig:ts-example}}
\end{figure}

Figure \ref{fig:ts-example} illustrates these phenomena as we vary the relative stockout costs
$K(a)$, and the remaining horizon. We show four panels where horizontally the horizon changes
and vertically the stock-out penalty $K$ changes. We observe that $K$ has a dramatic effect on
optimal inventory level (note that in this example ordering costs are zero, so the optimal
policy is \emph{only} driven by $K$ and $c$). Also note that the region where optimal policy is
$d^*(T,\vp) = 1$ is disjoint in the top two panels.

Figure \ref{fig:ts-regimes} again follows \cite{TreharneSox} and shows the effect of different
time dynamics of core process $M$. In the first case, we assume that demands are expected to
increase over time, so that the transition of $M$ follows the phases $1 \to 2 \to 3$. In that
case, it is possible that inventory will be increased even without any new events (i.e.\
$\tau_k \neq \sigma_\ell$). This happens because passage of time implies that the conditional
probability $\vP^{(3)}=\PP( M_t = 3|\F_t^{W})$ increases, and to counteract the corresponding increase in
probability of a stock-out, new inventory might be ordered. In the second case, we assume that
demand will be decreasing over time. In that case, the controller will order \emph{less}
compared to base case, since chances of overstocking will be increased.

\begin{figure}%[ht]\hspace{-1.5in}
\begin{tabular*}{\textwidth}{lr}
\begin{minipage}{3in}
\centering{\includegraphics[height=2.4in,width=3in,clip]{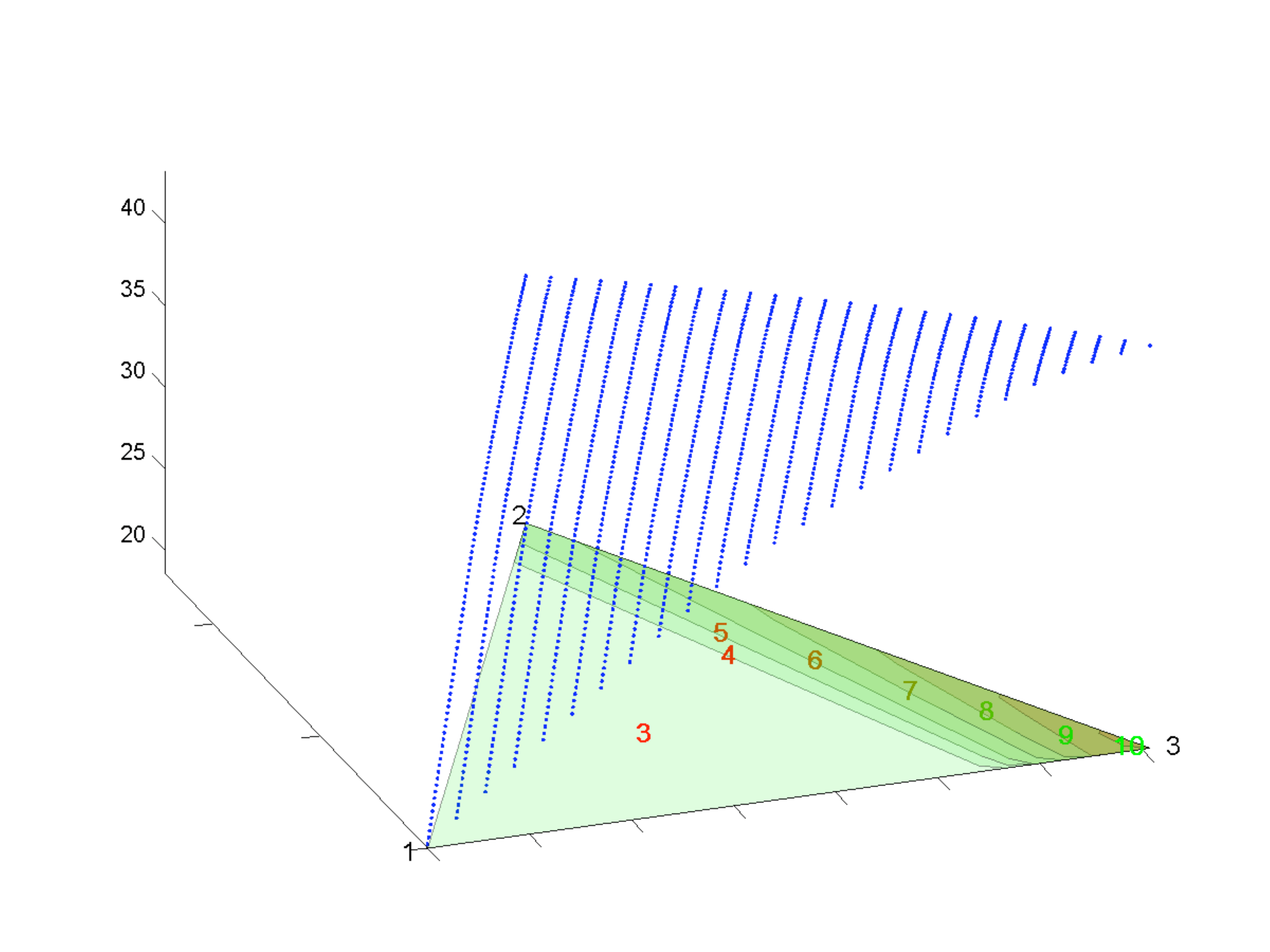}}

\end{minipage} &
\begin{minipage}{3in}
\centering{\includegraphics[height=2.4in,width=3in,clip]{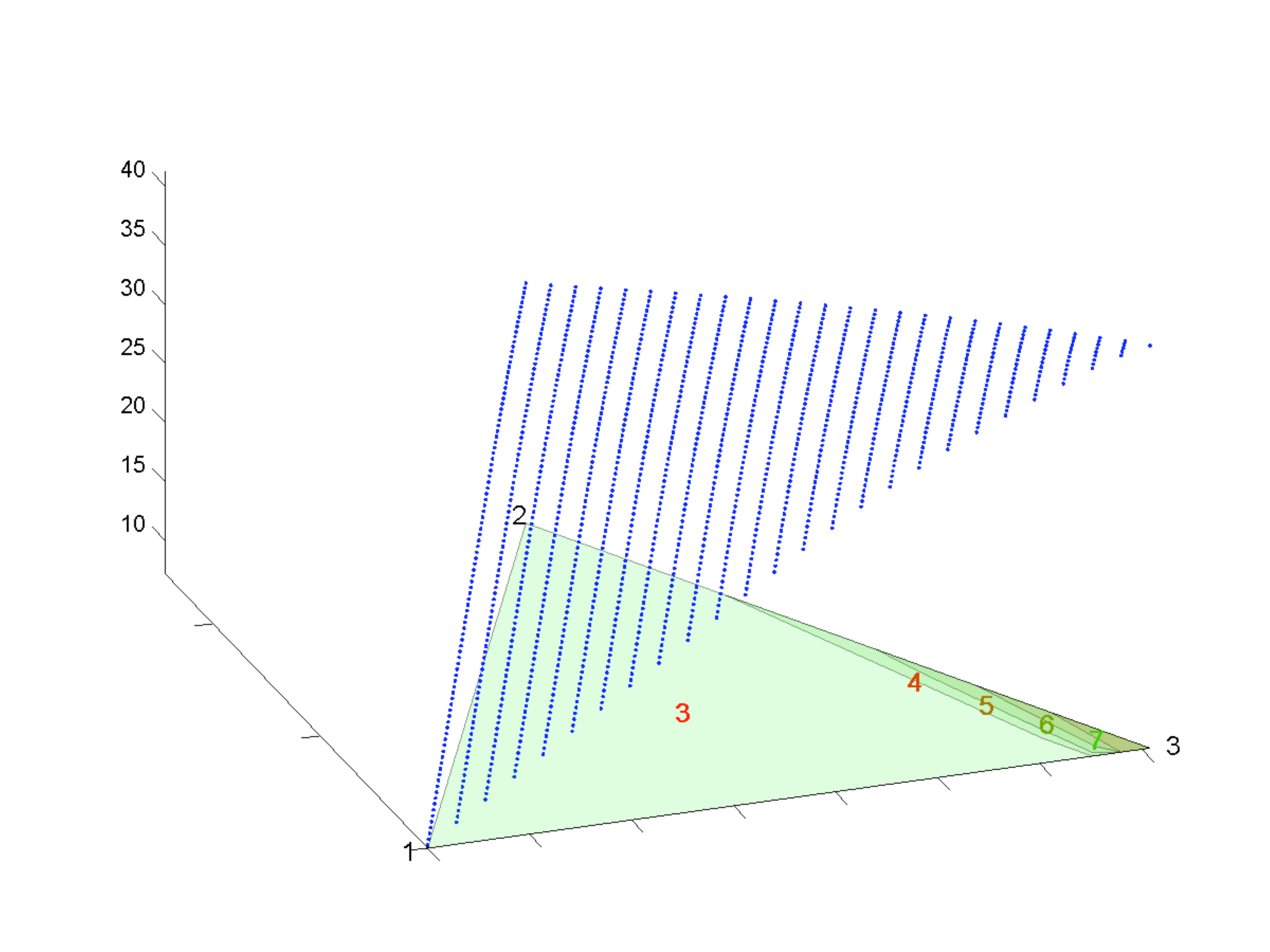}}

\end{minipage}
\end{tabular*}
\caption[Treharne-Sox Example]{Optimal inventory levels for different core process dynamics in
\cite{TreharneSox} (TS02) example of Section \ref{sec:ts-example}. The blue surfaces show
$U(T,\vp,3)$ over the triangle $\vp \in D=\{\pi_1+\pi_2+\pi_3 = 1\}$; underneath we show the
optimal inventory levels $d^*(T,\vp)$, see Figure \ref{fig:ts-example}.
Left panel (case US in TS02): $Q = \left(\begin{smallmatrix} -0.2 & 0.1 & 0.1 \\
0 & -0.2 & 0.2 \\
0 & 0 & 0 \end{smallmatrix} \right)$, right panel (case DS in TS02): $Q= \left(\begin{smallmatrix} 0 & 0 & 0 \\
0.2 & -0.2 & 0 \\
0.1 & 0.1 & -0.2 \end{smallmatrix} \right)$. (Note that in TS02 time is discrete. To be able to make a comparison we choose our generators to make the average holding time in each state equal to those of TS02.) \label{fig:ts-regimes}}
\end{figure}

\subsection{Example with Censoring}\label{sec:censoring-example} In our second example we
consider a model that treats censored observations. We assume that excess demand above
available stock is unobserved, and that a corresponding opportunity cost is incurred in case of
stock-out.

For parameters we choose
$$ Q = \begin{pmatrix}  -1 &  1 \\  1 &  -1 \end{pmatrix}, \qquad\vec\lambda = \begin{pmatrix}
2 \\ 1 \end{pmatrix}, \qquad \vec\nu = \begin{pmatrix} 0.5 & 0.4 & 0.1 \\
0.1 & 0.3 & 0.6 \end{pmatrix},
$$
so that demands are of size at most $R=3$. Note that in regime 2, demands are less frequent but
of larger size; also to distinguish between regimes it is crucial to observe the full demand
size.

The horizon is $T=3$ and costs are selected as
$$ c(a) = 2a, \qquad K(a) = 3.2a, \qquad h= 1.25, \qquad \zeta = 1,$$
with $\bar{P} = R = 3$. Again, we consider zero salvage value. These parameters have been
specially chosen to emphasize the effect of censoring.

We find that the effect of censoring on the value function $U$ is on the order of 3-4\% in this
example, see Table \ref{table:cens-ex} below. However, this obscures the fact that the optimal
policies are dramatically different in the two cases. Figure \ref{fig:censoring-regions}
compares the two optimal policies given that current inventory is empty. In general, as might
be expected, censored observations cause the manager to carry extra inventory in order to
obtain maximum information. However, counter-intuitively, there are also values of $t$ and
$\vp$ where censoring can lead to lower inventory (compared to no-censoring). We have observed
situations where censoring increases inventory costs by up to 15\%, which highlights the need
to properly model that feature (the particular example was included to showcase other features
we observe below).

\begin{figure}%[ht]\hspace{-1.5in}
\centering{\includegraphics[height=2.625in,width=7in,clip]{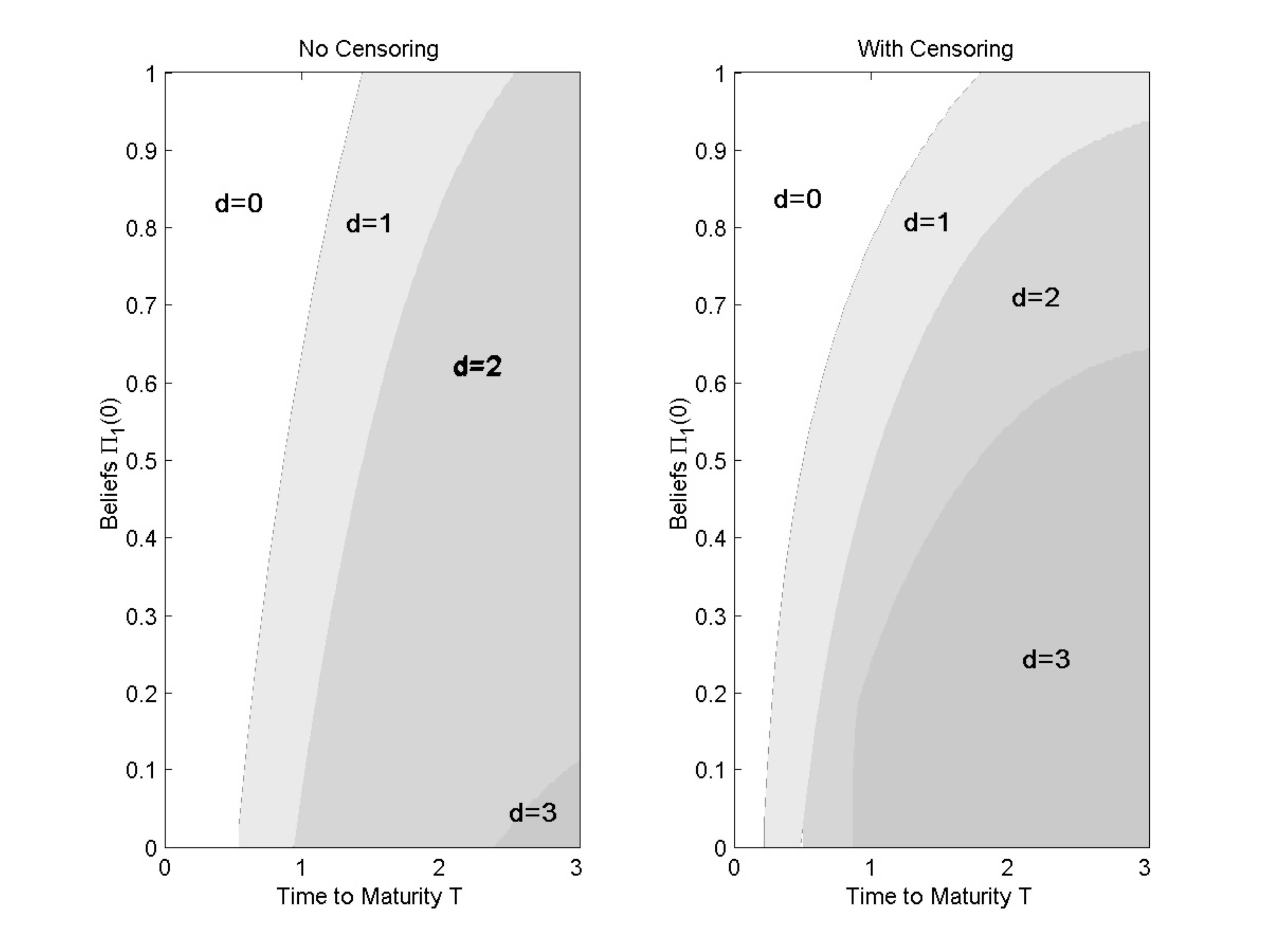}} \caption[Censoring
Example]{Optimal order levels as a function of time and beliefs given current empty inventory,
$P(0)=0$. We plot $d_{\M U}(T,\vp,0)$ as a function of time to maturity $T$ on the x-axis and
initial beliefs $\Pi_1(0) = \P(M_0=1)$ on the y-axis. \label{fig:censoring-regions}}
\end{figure}

\subsection{Optimal Strategy Implementation}

In Figure \ref{fig:optimal-strategy}, we present a sample path of the $(\vP, P)$-process which
shows the implementation of the optimal policy as defined in Figure
\ref{fig:censoring-regions}. We consider the above setting with censored observations, $T=3$
initial $\Pi_1(0) = 0.6$ (since $\Pi_2(t)\equiv 1-\Pi_1(t)$ in this one-dimensional example, we
focus on just the first component $\Pi_1(t) = \P(M_t=1)$) and initial zero inventory, $P(0)=0$.
In this example, it is optimal for the manager to place new orders only when inventory is
completely exhausted. Thus, $d_{\M U}(T,\vp,a)$ is non-trivial only for $a=0$; otherwise we
have $d_{\M U}(T,\vp,a)=a$ and the manager should just wait.

Since $d_{\M U}(3,\vP(0),0) = 3$, it is optimal for the manager to immediately put an order for
three units, as indicated by an arrow on the y-axis in Figure \ref{fig:optimal-strategy}. Then
the manager waits for demand orders, in the meantime paying storage costs on the three units on
inventory. At time $\sigma_1$, the first demand order (in this case of size two arrives). This
results in the update of the beliefs according to Proposition \eqref{cor:pdp} as
$$
\Pi_1(\sigma_1) = \frac{ \lambda_1 \cdot \nu_1(2) \cdot \Pi_1(\sigma_1-) }{\sum_{i=1}^2 \lambda_i
\cdot \nu_i(2) \cdot \Pi_i(\sigma_1-)} = \frac{2 \cdot 0.4 \cdot \Pi_1(\sigma_1-)}{2 \cdot 0.4
\cdot \Pi_1(\sigma_1-) + 1\cdot 0.3 \cdot (1-\Pi_1(\sigma_1-))}.
$$
\noindent This demand is fully observed and filled; since $d_{\M
U}(3-\sigma_1,\vP_{\sigma_1},1)=1$, no new orders are placed at that time. Then at time
$\sigma_2$ we assume that a censored demand (i.e.\ a demand of size more than 1 arrives). This
time the update in the beliefs is
$$
\Pi_1(\sigma_2) = \frac{ \lambda_1 \cdot [\nu_1(2)+\nu_1(3)] \cdot \Pi_1(\sigma_2-)
}{\sum_{i=1}^2 \lambda_i \cdot [\nu_i(2)+\nu_i(3)] \cdot \Pi_i(\sigma_2-)} = \frac{2 \cdot 0.5
\cdot \Pi_1(\sigma_2-)}{2 \cdot 0.5 \cdot \Pi_1(\sigma_2-) + 1\cdot 0.9 \cdot (1-\Pi_1(\sigma_2-))}.
$$
We see that the censored observation gives very little new information to the manager and
$\Pi_1(\sigma_2)$ is close to $\Pi_1(\sigma_2-)$. The inventory is now instantaneously brought
down to zero, as the one remaining unit is shipped out (and the rest is assigned an expected
lost opportunity cost). Because $d_{\M U}(T-\sigma_2, \vP(\sigma_2), 0)=1$, the manager
immediately orders one new unit of inventory, $\tau_1=\sigma_2$. Thus, overall we end up with
$P(\sigma_2) = 1$. At time $\sigma_3$, a single unit demand (uncensored) is observed; this
strongly suggests that $M_{\sigma_3}=1$ (due to short time between orders and a small order
amount), and $\Pi_1(\sigma_3)$ is indeed large. Given that little time remains till maturity,
it is now optimal to place no more new orders, $d_{\M U}(T-\sigma_3,\vP(\sigma_3),0)=0$.
However, as time elapses, $\P(M_t=2)$ grows and the manager begins to worry about incurring
excessive stock-out costs if a large order arrives. Accordingly, at time $\tau_2\sim 2.19$ (and
without new incoming orders) the manager places an order for a new unit of inventory as $(\vP,
P)$ again enters the region where $d_{\M U}=1$ (see lowermost panel of Figure
\ref{fig:optimal-strategy}). As it turns out in this sample, no new orders are in fact
forthcoming until $T$ and the manager will lose the latter inventory as there are no salvage
opportunities.

\begin{figure}%[ht]\hspace{-1.5in}
\centering{\includegraphics[height=3in,width=7.2in,clip]{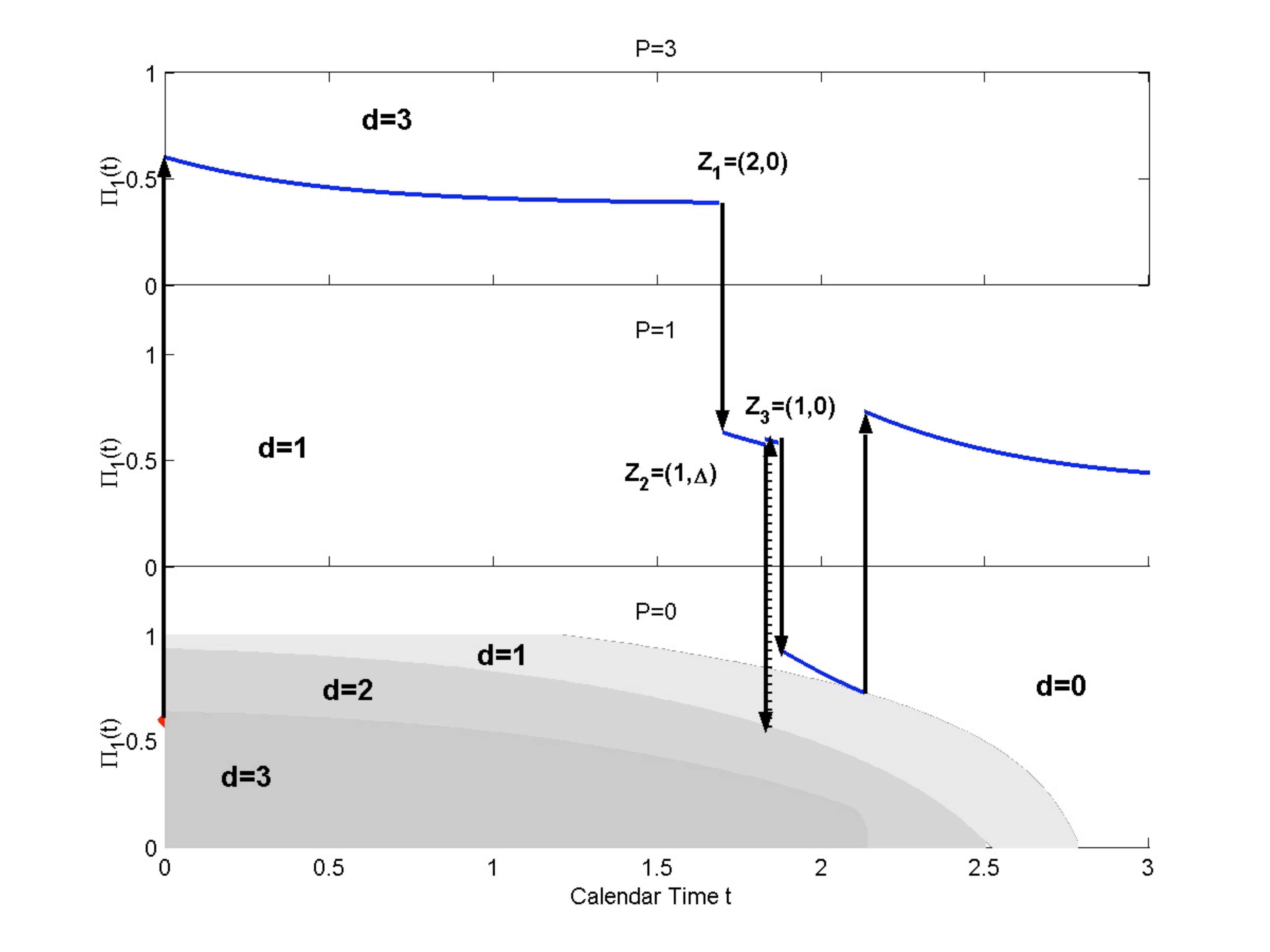}}
\caption{Implementation of optimal strategy on a sample path. We plot the beliefs
$\Pi_1(t)=\P(M_t=1)$ as a function of calendar time $t$, over the different panels
corresponding to $P(t)$. Incoming demand orders and placed supply orders are indicated with
circles and diamonds respectively. Here $T=3$ and $\vP(0) = [0.6, 0.4]$. The arrival times are
$\sigma_1=1.7, \sigma_2 = 1.83, \sigma_3 = 1.87$ and the corresponding marks are $Z_1 = (2,0),
Z_2=(1,\Delta), Z_3=(1,0)$, see Section \ref{sec:vP}. \label{fig:optimal-strategy}}
\end{figure}

\subsection{Effect of Other Parameters}\label{sec:param-effects}
We now compare the effects of other model parameters and ingredients on the value function and
optimal strategy. To give a concise summary of our findings, Table \ref{table:cens-ex} shows
the initial value function and initial optimal policy for one fixed choice of beliefs $\vP(0)$
and time-horizon $T$.

In particular, we compare the effect of censoring, as well as changes in various costs on the
value function and a representative optimal policy choice. We also study the effect of salvage
opportunities and possibility of selling inventory. Salvaging at x\% means that one adds the
initial condition $U(0,\vp,a) = - x \cdot h \cdot a$ to \eqref{def:U}, so that at the terminal
date the manager can recover some of the unit costs associated with unsold inventory. Selling
inventory means that at any point in time the manager can sell back unneeded inventory, so that
$\xi_k \in \{-P(\tau_k), \ldots, \bar{P}-P(\tau_k)\}$ and the minimization in
\eqref{eq:intervention-op} is over all $b \in [0,\bar{P}]$.

As expected, storage costs increase average costs and cause the manager to carry less
inventory; conversely stock-out costs cause the manager to carry more inventory (and also
increase the value function). Fixed order costs are also crucial and increase supply order
sizes, as each supply order is very costly to place. We find that in this example, the
possibility of salvaging inventory and opportunity to sell inventory have little impact on the
value function, which appears to be primarily driven by potential stock-out penalties.

\begin{table}
\begin{tabular}{l|c|c}
Model  & Value Function  $U(T, \vp, 0)$ & Optimal Policy $d_{\M U}(T,\vp,0)$ \\
\hline
Base case w/out censoring & 25.21 & 2 \\
Base case w/censoring &  25.97 & 3 \\
Reduced stockout costs $K(a)=2a$ & 16.37 & 0 \\
Zero storage costs $c=0$ & 16.71  & 3 \\
No fixed ordering costs & 22.92 & 1 \\
Terminal salvage value of 50\% &  24.75 & 2 \\
Can buy/sell at cost, $\zeta=1$ & 24.30  & 2 \\
\end{tabular}
\caption{Results for the comparative statics in Section \ref{sec:param-effects}. Here $T=3$ and
$\vP(0)=\vp=(0.5,0.5)$, so that $\P(M_0=1)=\P(M_0=2)=\frac{1}{2}$.
\label{table:cens-ex}}\end{table}

\section{Conclusion}
In this paper we have presented a probabilistic model that can address the setting of partially
observed non-stationary demand in inventory management. Applying the DP intuition, we have
derived and gave a direct proof the dynamic programming equation for the value function and the
ensuing characterization of an optimal policy. We have also derived an alternative
characterization that can be easily implemented on a computer. This gives a method to compute
the full value function/optimal policy to any degree of accuracy. As such, our model contrasts
with other approaches that only present heuristic policy choices.

Our model can also incorporate demand censoring. Our numerical investigations suggest that
censored demands may have a significant influence on the optimal value to extract and even a
more dramatic impact on the manager's optimal policies. This highlights the need to properly
model demand censoring in applications. It would be of interest to further study this aspect of
the model and to compare it with real-life experiences of inventory managers.

\appendix
\renewcommand{\thesection}{A}
\refstepcounter{section}
\makeatletter
\renewcommand{\theequation}{\thesection.\@arabic\c@equation}
\makeatother

\section{Proof of Proposition~\ref{sec:pfpdp}}

\subsection{A Preliminary Result}

\begin{lemm}
\label{lemm:vP-explicit} For $i \in E$, let us define
\begin{align}
\label{def:L} L_i^{\vp,a,\xi} ( t, r : (t_k, \vz_k ), k \le r) \triangleq \E^{\vp,a,\xi} \left[
1_{\{ M_t =i\}} \cdot\e^{ - I(t) } \cdot  \prod_{k=1}^r \ell(t_k, \vz_k) \right],
\end{align}
where
\begin{align}
\label{def:I} \ell(t, \vz) \triangleq \sum_{j \in E}  1_{\{ M_{t} =j \}} \lambda_j \cdot g_j
(\vz; P_{t-}).
\end{align}
Moreover let $L^{\vp,a,\xi} ( t, r : (t_k, \vz_k ), k \le r) = \sum_{j \in E} L_j^{\vp,a,\xi} (
t, r : (t_k, \vz_k ), k \le r)$. Then we have
\begin{align}
\label{vP-explicit} \Pi_i(t) &= \frac{L_i^{\vp,a,\xi} ( t, N_t : (\sigma_k, Z_k ), k \le N_t)
}{ L^{\vp} ( t, N_t : (\sigma_k, Z_k ), k \le N_t)} \equiv \left[  \frac{L_i^{\vp,a,\xi} ( t, r
: (t_k, \vz_k ), k \le r) }{ L^{\vp,a,\xi} ( t, r : (t_k, \vz_k ), k \le r)} \right]\Bigg|_{ r
= N_t \, ; \,  ( t_k=\sigma_k, \vz_k = Z_k )_{ k \le r} }
\end{align}
$ \P^{\vp}$-a.s., for all $t\ge 0$, and for $i  \in E$.
\end{lemm}

\begin{proof}
Let $\Xi$ be a set of the form
\begin{align*}
{\Xi} = \{  N_{t_1} = r_1, \ldots, N_{t_k} = r_k ; (Z_1,\ldots, Z_{r_k}) \in B  \}
\end{align*}
where $0 = t_0 \le t_1 \le \ldots \le t_k = t $ with $0 \le r_1 \le \ldots \le r_k$ for $k \in
\N$, and $B$ is a Borel set in $\mathcal{B}(\R^{r_k})$. Since $t_j$ and $r_j$'s are arbitrary,
to prove (\ref{vP-explicit}) by the Monotone Class Theorem it is then sufficient to establish
\begin{align*}
\E^{\vp,a,\xi} \left[ 1_{\Xi} \cdot \P^{\vp} \{ M_t =i | \Fc^W_t \} \right] = \E^{\vp,a,\xi}
\left[ 1_{\Xi} \cdot \frac{L_i^{\vp,a,\xi} ( t, N_t : (\sigma_k, Z_k ), k \le N_t) }{
L^{\vp,a,\xi} ( t, N_t : (\sigma_k, Z_k ), k \le N_t)} \right].
\end{align*}
Conditioning on the path of $M$, the left-hand side (LHS) above equals
{\small
\begin{multline*}
LHS  = \E^{\vp,a,\xi} \left[ 1_{\{M_t = i\}} \P^{\vp} \left\{ N_{t_1} = r_1, \ldots, N_{t_k} = r_k ; (Z_1, \ldots, Z_{r_k}) \in B \, \Big| M_s , \, s \le t\ \right\}  \right] =\\
 \E^{\vp,a,\xi} \left[ 1_{\{M_t = i\}}  \int_{B \times \Upsilon (t_1, \dotsc, t_k)} \P^{\vp} \left\{ \sigma_1 \in ds_1, \ldots, \sigma_{r_k} \in d s_{r_k} ; Z_1 = \vz_1, \ldots, Z_{r_k} = \vz_{r_k} \, \Big| M_s, \, s \le t\ \right\}  \right]
\end{multline*}}
where
\begin{align*}
\Upsilon (t_1, \dotsc, t_k) = \left\{ s_1, \ldots, s_{r_k} \in \R_+^{r_k}: \; s_1\le \ldots \le
s_{r_k} \le t \; \; \text{and}\; \; s_{r_j} \le t_j < s_{r_j+1} \; \text{for}\;  j=1, \ldots k
\right\}.
\end{align*}
Then, using \eqref{path-likelihood-given-M} and Fubini's theorem we obtain
\begin{align*}
LHS & =\E^{\vp,a,\xi} \left[ 1_{\{M_t = i\}}\int_{B \times \Upsilon (t_1, \dotsc, t_k)} \e^{-I(t)}  \prod_{\ell=1}^{r_k} \sum_{j \in E }  1_{\{M_{s_\ell} = j\}} \lambda_j g_j(\vz_\ell; P_{s_\ell}) \,  ds_\ell \,   \right] \\
& =\int_{B \times \Upsilon (t_1, \dotsc, t_k)} \E^{\vp,a,\xi} \left[ 1_{\{M_t = i\}} \e^{-I(t)}  \prod_{\ell=1}^{r_k} \sum_{j \in E }  1_{\{M_{s_\ell} = j\}} \lambda_j g_j(\vz_\ell; P_{s_\ell})   \right] \prod_{\ell=1}^{r_k}ds_\ell \\
& =\int_{B \times \Upsilon (t_1, \dotsc, t_k)}  L_i^{\vp,a,\xi} ( t, r_k : (s_j, \vz_j ), j \le r_k) \prod_{\ell=1}^{r_k}ds_\ell \\
& =\int_{B \times \Upsilon (t_1, \dotsc, t_k)} \frac{L_i^{\vp,a,\xi} ( t, r_k : (s_j, \vz_j ), j \le r_k) }{ L^{\vp,a,\xi} ( t, r_k : (\sigma_j, Y_j ), j \le r_k) } \cdot  L^{\vp,a,\xi} ( t, r_k : (s_j, \vz_j ), j \le r_k) \prod_{\ell=1}^{r_k}ds_\ell \\
& =\int_{B \times \Upsilon (t_1, \dotsc, t_k)} \frac{L_i^{\vp,a,\xi} ( \ldots) }{ L^{\vp,a,\xi}
( \ldots)} \cdot     \E^{\vp} \left[ \sum_{i \in E }1_{\{M_t = i\}} \e^{-I(t)}
\prod_{\ell=1}^{r_k} \sum_{j \in E }  1_{\{M_{s_\ell} = j\}} \lambda_j g_j(\vz_l; P_{s_\ell})
\right] \prod_{\ell=1}^{r_k}ds_\ell.
\end{align*}
By another application of Fubini's theorem, we obtain $LHS =$
\begin{align*}
&\E^{\vp,a,\xi} \left[ \sum_{i \in E }1_{\{M_t = i\}}  \int_{B \times \Upsilon (t_1, \dotsc,
t_k)} \frac{L_i^{\vp,a,\xi} ( t, r_k : (s_j, \vz_j ), j \le r_k) }{ L^{\vp,a,\xi} ( t, r_k :
(\sigma_j, \vz_j ), j \le r_k) } \cdot    \e^{-I(t)}  \prod_{\ell=1}^{r_k} \sum_{j \in E }
1_{\{M_{s_\ell} = i\}} \lambda_j g_j(\vz_\ell; P_{s_\ell})
\cdot \prod_{\ell=1}^{r_k}ds_\ell \right] \\
&=\E^{\vp,a,\xi} \left[ \sum_{i \in E }1_{\{M_t = i\}}  \E^{\vp,a,\xi} \left[ 1_{\{  N_{t_1} =
r_1, \ldots, N_{t_k} = r_k ; (Z_1,\ldots, Z_{r_k}) \in B  \}} \cdot \frac{L_i^{\vp,a,\xi} ( t,
N_t : (\sigma_j, Z_j ), j \le N_t) }{ L^{\vp,a,\xi} ( t, N_t : (\sigma_j, Z_j ), j \le N_t) }
\;
\Bigg| M_s ; \, s \le t \right] \right] \\
&=\E^{\vp,a, \xi} \left[   \E^{\vp,a,\xi} \left[ 1_{\{  N_{t_1} = r_1, \ldots, N_{t_k} = r_k ;
(Z_1,\ldots, Z_{r_k}) \in B  \}} \cdot \frac{L_i^{\vp,a,\xi} ( t, N_t : (\sigma_j, Z_j ), j \le
N_t) }{ L^{\vp,a,\xi} ( t, N_t : (\sigma_j, Z_j ), j \le N_t) }
\; \Bigg| M_s ; \, s \le t \right] \right] \\
&=\E^{\vp,a,\xi} \left[   1_{\Xi} \cdot \frac{L_i^{\vp,a,\xi} ( t, N_t : (\sigma_j, Z_j ), j
\le N_t) }{ L^{\vp,a,\xi} ( t, N_t : (\sigma_j, Z_j ), j \le N_t) }
 \right].
\end{align*}
\end{proof}
\subsection{Proof of Proposition~\ref{cor:pdp} }\label{sec:pfpdp}

Let $\E^{a,\xi}_j[\cdot]$ denote the expectation operator $\E^{\vp,a,\xi}[\cdot \, | M_0 =j] $,
and let $t_r \le t \le t+ u $, then
\begin{multline}
\label{eq:derivation-of-vx} L_i^{\vp,a,\xi} ( t+u  , r : (t_k, \vz_k ), k \le r) = \sum_{j \in
E} \pi_j \cdot \E^{a,\xi}_j \left[ 1_{\{ M_{t+u} =i\}} \cdot\e^{ - I(t+u) } \cdot \prod_{k=1}^r
\ell(t_k, \vz_k)
\right] \\
\begin{aligned}
=& \sum_{j \in E} \pi_j  \cdot \E^{a,\xi}_j \left[ \E^{a,\xi}_j \Biggl[ 1_{\{ M_{t+u} =i\}}
\cdot\e^{ - I(t+u) } \cdot  \prod_{k=1}^r \ell(t_k, \vz_k) \,
\Bigg| M_s, s\le t\Biggr]   \right] \\
=& \sum_{j \in E} \pi_j  \cdot \E^{a,\xi}_j \left[\e^{ - I(t) } \left( \prod_{k=1}^r \ell(t_k,
\vz_k) \right) \E^{a,\xi}_j \left[ 1_{\{ M_{t+u} = i \}} \cdot\e^{ - (I(t+u)-I(t) )}  \Bigg|
M_t \right] \right] ,
\end{aligned}
\end{multline}
where the third equality followed from the properties of the conditional expectation and the Markov property of $M$.
The last expression in \eqref{eq:derivation-of-vx} can be
written as
\begin{multline*}
\begin{aligned}
 =& \sum_{j \in E} \pi_j
\cdot \E^{a,\xi}_j \left[\e^{ - I(t) } \left( \prod_{k=1}^r \ell(t_k, \vz_k) \right) \cdot
\sum_{l \in E} 1_{\{ M_{t} =l\}} \cdot \E^{a,\xi}_l \left[
1_{\{ M_{u} =i\}} \cdot\e^{ - I(u)}   \right]   \right] \\
=& \sum_{l \in E}  \E^{a,\xi}_l \left[ 1_{\{ M_{u} =i\}} \cdot\e^{ - I(u)}   \right] \cdot
\E^{\vp} \left[ 1_{\{ M_{t} =l\}}  \cdot\e^{ - I(t) } \prod_{k=1}^r \ell(t_k, \vz_k)
 \right] \\
=& \sum_{l \in E}  \E^{a,\xi}_l \left[ 1_{\{ M_{u} =i\}} \cdot\e^{ - I(u)}   \right] \cdot
L_l^{\vp} ( t  , r : (t_k, \vz_k ), k \le r).
\end{aligned}
\end{multline*}
Then the explicit form of $\vP$ in (\ref{vP-explicit}) implies that for $\sigma_m \le t \le t+
u < \sigma_{m+1}$, we have
\begin{align}
\label{eq:semi-group}
\begin{aligned}
\Pi_i (t+u) &= \frac{\sum_{l \in E}  L_l^{\vp} ( t  , r : (\sigma_k, \vz_k ), k \le r) \cdot
\E^{a,\xi}_l \left[ 1_{\{ M_{u} =i\}} \cdot\e^{ - I(u)}   \right]  } {\sum_{j \in E} \sum_{l
\in E} L_{l}^{\vp} ( t  , r : (\sigma_k, \vz_k ), k \le r) \cdot  \E^{a,\xi}_l \left[
1_{\{ M_{u} =j\}} \cdot\e^{ - I(u)}   \right] } \\
&= \frac{\sum_{l \in E}  \Pi_l(t) \cdot \E^{a,\xi}_l \left[ 1_{\{ M_{u} =i\}} \cdot\e^{ - I(u)} \right]
} {\sum_{j \in E} \sum_{l \in E}  \Pi_{l}(t) \cdot  \E^{a,\xi}_l \left[ 1_{\{ M_{u} =j\}} \cdot\e^{ -
I(u)}   \right] } =  \frac{  \E^{\vP_t} \left[ 1_{\{ M_{u} =i\}} \cdot\e^{ - I(u)} \right]  }
{\sum_{j \in E}  \E^{\vP_t} \left[
1_{\{ M_{u} =j\}} \cdot\e^{ - I(u)}   \right] } \\
&=  \frac{  \E^{\vP_t} \left[ 1_{\{ M_{u} =i\}} \cdot\e^{ - I(u)}   \right]  } {\E^{ \vP_t }
\left[  \e^{ - I(u)}   \right] } = \frac{    \P^{\vp}  \{ \sigma_1 > u , M_u =i  \}  }
 {   \P^{\vp}  \{ \sigma_1 > u  \} }\Bigg|_{\vp = \vP_t}.
\end{aligned}
\end{align}
On the other hand, the expression in (\ref{def:L}) implies
\begin{multline}
\label{eq:derivation-of-jump-part} L_i^{\vp,a,\xi} ( \sigma_{r+1} , r+1 : (\sigma_k, Z_k ), k
\le r+1)
 = \E^{\vp,a,\xi} \left[
1_{\{ M_t =i\}} \cdot\e^{ - I(t) } \cdot  \prod_{k=1}^{r+1} \ell(t_k, \vz_k)
 \right] \Bigg|_{ t = \sigma_{r+1} \, ; \,  (t_k=\sigma_k, y_k = Z_k )_{k \le r+1}   } \\
\begin{aligned}
= \lambda_i g_i(Z_{r+1}; P_{\sigma_{r+1}-}) \cdot \E^{\vp,a,\xi} \left[ 1_{\{ M_t =i\}}
\cdot\e^{ - I(t) } \cdot \prod_{k=1}^r \ell(t_k, \vz_k) \right] \Bigg|_{ t = \sigma_{r+1} \, ;
\, (t_k=\sigma_k, \vz_k = Z_k )_{ k \le r } }.
\end{aligned}
\end{multline}
Note that for fixed time $t$, we have $M_t = M_{t-}$, $\P^{\vp,a,\xi}$-a.s. and $L_i^{\vp,a,xi}
(t , r : (t_k, \vz_k ), k \le r)= L_i^{\vp,a,\xi} (t- , r : (t_k, \vz_k ), k \le r)$ when $t_r
< t$. Then we obtain
\begin{align}\label{eq:jumps-of-vP}
L_i^{\vp,a,\xi} ( \sigma_{r+1} , r+1 : (\sigma_k, Z_k ), k \le r+1) = \lambda_i g_i(Z_{r+1};
P_{\sigma_{r+1}-}) \cdot L_i^{\vp,a,\xi} ( \sigma_{r+1}- , r : (\sigma_k, Z_k ), k \le r),
\end{align}
due to \eqref{eq:derivation-of-jump-part}. Hence we conclude that at arrival times $\sigma_1,
\sigma_2, \ldots$ of $Z$, the process $\vP$ exhibits a jump behavior and satisfies the
recursive relation
\begin{align*}
%\label{eq:jumps-of-vP}
\begin{aligned}
\Pi_i (\sigma_{r+1} ) &= \frac{ \lambda_i g_i(Z_{r+1}; P_{\sigma_{r+1}-})  L_i^{\vp} (
\sigma_{r+1}- , r : (\sigma_k, Z_k ), i \le r) } { \sum_{j \in E} \lambda_j g_j(Z_{r+1};
P_{\sigma_{r+1}-}) L_j^{\vp} ( \sigma_{r+1}- , r : (\sigma_k, Z_k ), k \le r) }  \\
& =\frac{ \lambda_i g_i(Z_{r+1}; P_{\sigma_{r+1}-}) \Pi_i (\sigma_{r+1} - ) }
{ \sum_{j \in E} \lambda_j g_j(Z_{r+1}; P_{\sigma_{r+1}-})  \Pi_j (\sigma_{r+1} - ) }%, \qquad \text{for $m \in \N $,}
\end{aligned}
\end{align*}
for $r \in \N $.

\appendix
\renewcommand{\thesection}{B}
\refstepcounter{section}
\makeatletter
\renewcommand{\theequation}{\thesection.\@arabic\c@equation}
\makeatother

\section{Analysis Leading to the Proof of Proposition~\ref{prop:dp}}\label{sec:analysis}

\subsection{Preliminaries}\label{sec:analysis1}

Let us consider the following restricted version of (\ref{def:U})
\begin{equation}\label{eq:aux}
U_n(T,\vp,a)\triangleq
 \inf_{ \xi \in \U_n(T)} J^{\xi}(T,\vp,a), \qquad n \ge 1,
\end{equation}
in which $\U_n(T)$ is a subset of $\U(T)$ which contains strategies with at most $n \ge 1$
supply orders up to time $T$.

The following proposition shows that the value functions $(U_n)_{n \in \mathbb{N}}$ of
\eqref{eq:aux} which correspond to the restricted control problems over $\U_n(T)$ can be
alternatively obtained via the sequence of iterated optimal stopping problems in
\eqref{defn:Un}.
\begin{prop}\label{prop:UneqVn}
$U_n=V_n$ for $n \in \mathbb{N}$.
\end{prop}
\begin{proof}
By definition we have that $U_0=V_0$. Let us assume that $U_n=V_n$ and show that
$U_{n+1}=V_{n+1}$. We will carry out the proof in two steps.

\textbf{Step 1}. First we will show that  $U_{n+1} \geq V_{n+1}$. Let $\xi \in \U_{n+1}(T)$,
\begin{align*}
\xi_t = \sum_{k=0}^{n+1} \xi_k \cdot 1_{[\tau_k, \tau_{k+1})}(t), \quad t \in [0,T],
\end{align*}
with $\tau_0=0$ and $\tau_{n+1}=T$, be $\eps$-optimal for the problem in (\ref{eq:aux}), i.e.,
\begin{equation}\label{eq:eps-opt}
U_{n+1}(T,\vp,a)+\eps \geq J^{\xi}(T,\vp,a).
\end{equation}
Let $\tilde{\xi} \in \U_n(T)$ be defined as
$\tilde{\tau}_k=\tau_{k+1}$ , $\tilde{\xi}_k=\xi_{k+1}$,  for $k \in \mathbb{N}_+$. Using the
strong Markov property of $(\vP,P)$, we can write $J^{\xi}$ as
\begin{equation}\label{eq:eps-argument}
\begin{split}
J^{\xi}(T,\vp,a)&= \E^{\vp,a}\Bigl[\int_0^{\tau_1}\e^{-\rho s} c(P_s)\, ds+ \sum_\ell
\e^{-\sigma_\ell} 1_{\{ \sigma_\ell \le \tau_1 \} }K( (Y_\ell - P_{\sigma_\ell-})_+)  \\ &
\qquad + \e^{-\rho
\tau_1}\left(J^{\tilde{\xi}}(T-\tau_1,\vP_{\tau_1},P_{\tau_1-}+\xi_1)+h \cdot \xi_1+\zeta\right)\Bigr]
\\ &\geq \E^{\vp,a}\Bigl[\int_0^{\tau_1}\e^{-\rho s} c(P_s) \, ds+ \sum_\ell \e^{-\sigma_\ell}1_{\{ \sigma_\ell \le \tau_1 \} }K( (Y_\ell - P_{\sigma_\ell-})_+)  \\ &
\qquad \e^{-\rho
\tau_1}\left(V_{n}(T-\tau_1,\vP_{\tau_1},P_{\tau_1-}+\xi_1)+h \cdot \xi_1+\zeta\right)\Bigr]
\\ &\geq \E^{\vp,a}\Bigl[\int_0^{\tau_1}\e^{-\rho s} c(P_s) \, ds+ \sum_\ell \e^{-\sigma_\ell}1_{\{ \sigma_\ell \le \tau_1 \} }K( (Y_\ell - P_{\sigma_\ell-})_+)  \\ &
\qquad \e^{-\rho \tau_1} \M V_n  (T-\tau_1, \vP_{\tau_1},P_{\tau_1-})\Bigr]
\\ &\geq \G V_{n}(T,\vp,a)=V_{n+1}(T,\vp,a).
\end{split}
\end{equation}
Here, the first inequality follows from induction hypothesis, the second inequality follows
from the definition of $\M$, and the last inequality from the definition of $\G$. As a result
of (\ref{eq:eps-opt}) and (\ref{eq:eps-argument}) we have that $U_{n+1} \geq V_{n+1}$ since
$\eps>0$ is arbitrary.

\textbf{Step 2}. To show the opposite inequality $U_{n+1} \leq V_{n+1}$, we will construct a
special $\bar{\xi} \in \U_{n+1}(T)$. To this end let us introduce
\begin{align}\left\{ \begin{aligned}
\bar{\tau}_1 & = \inf\{t\geq 0: \M V_n(T-t,\vP_t, P_t) \leq V_{n+1}(T-t,\vP_t,P_t)+\eps \}, \\
\bar{\xi}_1 & = d_{\M V_n}(T-\bar{\tau}_1,\vP_{\bar{\tau}_1},P_{\bar{\tau}_1-}).
\end{aligned}\right.
\end{align}
\noindent  Let $\hat{\xi}_t = \sum_{k=0}^{n} \hat{\xi}_k \cdot 1_{[\hat{\tau}_k,
\hat{\tau}_{k+1})}(t)$, $\hat{\xi} \in \U_n(T)$ be $\eps$-optimal for the problem in which $n$
interventions are allowed, i.e.\ (\ref{eq:aux}). Using $\hat{\xi}$ we now complete the
description of the control $\bar{\xi} \in \U_{n+1}(T)$ by assigning,
\begin{equation}
\bar{\tau}_{k+1}=\hat{\tau}_k \circ \theta_{\tau_1}, \quad \bar{\xi}_{k+1}=\hat{\xi}_{k} \circ
\theta_{\bar{\tau}_1}, \quad k \in \mathbb{N}_+,
\end{equation}
in which $\theta$ is the classical \emph{shift operator} used in the theory of Markov processes.

Note that $\bar{\tau}_1$ is an $\eps$-optimal stopping time for the stopping problem in the
definition of $\G V_n$. This follows from the classical optimal stopping theory since the
process $(\vP, P)$ has the strong Markov property. Therefore,
\begin{equation}\label{eq:tos}
\begin{split}
V_{n+1}(T,\vp,a)+\eps &\geq \E^{\vp,a}\Bigl[\int_0^{\bar{\tau}_1}\e^{-\rho s}c(P_s) \, ds+
\sum_\ell \e^{-\sigma_\ell}1_{\{ \sigma_\ell \le \tau_1 \} }K( (Y_\ell - P_{\sigma_\ell-})_+)
\\ & \qquad +\e^{-\rho \bar{\tau}_1}\M V_{n}(T-\bar{\tau}_1,\vP_{\bar{\tau}_1},P_{\bar{\tau}_1-})\Bigr]
\\& \geq  \E^{\vp,a}\Bigl[\int_0^{\bar{\tau}_1}\e^{-\rho s}c(P_s)\, ds+ \sum_\ell \e^{-\sigma_\ell}1_{\{ \sigma_\ell \le \tau_1 \} }K( (Y_\ell - P_{\sigma_\ell-})_+)
\\ & \qquad + \e^{-\rho \bar{\tau}_1} \left(U_{n}(T-\bar{\tau}_1,\vP_{\bar{\tau}_1}, P_{\bar{\tau}_1-} + \bar{\xi}_1) + h \cdot \bar{\xi}_1+\zeta\right)\Bigr],
\end{split}
\end{equation}
in which the second inequality follows from the definition of $\bar{\xi}_1$ and the induction hypothesis.
It follows from (\ref{eq:tos}) and the strong Markov property of $(\vP,\xi)$ that
\begin{equation}
\begin{split}
V_{n+1}(T,\vp,a)+2\eps &\geq \E^{\vp,a}\Bigl[\int_0^{\bar{\tau}_1}\e^{-\rho s}c(P_s)\,ds+
\sum_\ell \e^{-\sigma_\ell}1_{\{ \sigma_\ell \le \tau_1 \} }K( (Y_\ell - P_{\sigma_\ell-})_+)
\\ & \qquad + \e^{-\rho \bar{\tau}_1}
\left(U_{n}(T-\bar{\tau}_1,\vP_{\bar{\tau}_1},P_{\bar{\tau}_1-} + \bar{\xi}_1)+\eps
+h \cdot \bar{\xi}_1+\zeta\right)\Bigr]
\\&\geq \E^{\vp,a}\Bigl[\int_0^{\bar{\tau}_1}\e^{-\rho s}c(P_s)\, ds+ \sum_\ell \e^{-\sigma_\ell}1_{\{ \sigma_\ell \le \tau_1 \} }K( (Y_\ell - P_{\sigma_\ell-})_+)
\\ & \qquad + \e^{-\rho \bar{\tau}_1} \left(J^{\hat{\xi}}(T-\bar{\tau}_1,\vP_{\bar{\tau}_1}, P_{\bar{\tau_1}-}+\xi_1)+h \cdot \bar{\xi}_1+\zeta\right)\Bigr]
\\&=J^{\bar{\xi}}(T,\vp,a) \leq U_{n+1}(T,\vp,a).
\end{split}
\end{equation}
This completes the proof of the second step since $\eps>0$ is arbitrary.
\end{proof}

\subsection{Proof of ~\ref{prop:vnconvU}}\label{sec:proofofconvprop}
Let us denote $V(T,\vp,a) \triangleq \lim_{n \rightarrow \infty}V_n(T,\vp,a)$, which is well-defined thanks to the monotonicity of $(V_n)_{n \in \N}$. Since $\U_n(T) \subset \U(T)$, it follows that
$V_n(T,\vp,a)=U_n(T,\vp,a) \geq U(T, \vp, a)$. Therefore $V(T,\vp,a) \geq U(T,\vp,a)$. In the
remainder of the proof we will show that $V(T,\vp,a) \leq U(T,\vp,a)$.

Let $\xi \in \U(T)$ and  $\tilde{\xi}_t \triangleq  \xi_{t \wedge \tau_n}$. Observe that $\tilde{\xi} \in
\U_n(T)$. Then
\begin{align} \notag
|J^{\xi}(T,\vp,a) & -J^{\tilde{\xi}}(T,\vp,a)|  \\  \label{eq:int-step}&  \leq
\E^{\vp,a,\xi}\bigg[\int_{\tau_n}^{T}\e^{-\rho s}|c(P_s(\xi))-c(P_s(\tilde{\xi}))| \, ds + \sum_{k
\geq n+1} \e^{-\rho \tau_k}(h \cdot \xi_{k} +\zeta)
\\ \notag & \quad + \sum_\ell [K( (Y_\ell -
P_s(\xi))_+) + K( (Y_\ell - P_s(\tilde{\xi}))_+)]1_{\tau_n < \sigma_\ell < T} \bigg]\\ \notag &
\leq 2c(\bar{P}) \, \E^{\vp,a,\xi}\left[\int_{\tau_n}^{T}\e^{-\rho s}ds\right]+2
K(R)\E^{\vp,a,\xi}\left[\sum_{\ell} 1_{\tau_n < \sigma_\ell < T} + \sum_{k \geq n+1}
\e^{-\rho \tau_k} (h \cdot \xi_k+\zeta)\right],
\end{align}
\noindent Now, the right-hand-side of (\ref{eq:int-step}) converges to 0 as $n \rightarrow
\infty$. Since there are only finitely many switches almost surely for any given path,
\[
\lim_{n \rightarrow \infty}\int_{0}^{T}1_{\{s>\tau_n\}}\e^{-\rho s}\, ds + \sum_{k \geq n+1}
\e^{-\rho \tau_k} (h \cdot \xi_k+\zeta)=0,
\]
The admissibility condition \eqref{eq:admissiblity} along with the dominated convergence theorem implies that
\[
\lim_{n \rightarrow \infty}\E^{\vp,a,\xi}\left[\int_{\tau_n}^{T}\e^{-\rho s} ds + \sum_{k \geq n+1}
\e^{-\rho \tau_k} (h \cdot \xi_k+\zeta)\right]=0.
\]

On the other hand,
 $$ \E^{\vp,a,\xi}[ \sum_{\ell} 1_{\sigma_\ell < T} ] = \E^{\vp,a,\xi}[ N(T) ] < \infty, $$ (see e.g. the estimate in \eqref{eq:est-lst}). Therefore by
the monotone convergence theorem
$$ \lim_{n \rightarrow \infty} \E^{\vp,a,\xi}[ \sum_{\ell} 1_{\tau_n < \sigma_\ell < T} ] = 0.$$

As a result, for any $\epsilon > 0$ and $n$ large enough, we find
\begin{equation*}
|J^{\xi}(T,\vp,a)-J^{\tilde{\xi}}(T,\vp,a)| \leq \eps.
\end{equation*}
Now, since $\tilde{\xi} \in \U_n(T)$ we have $V_n(T,\vp,a) =U_{n}(T,\vp,a)\leq
J^{\tilde{\xi}}(T,\vp,a) \leq J^{\xi}(T,\vp,a)+\eps$ for sufficiently large $n$, and it follows
that
\begin{equation}\label{eq:WgewJ}
V(T,\vp,a)=\lim_{n \rightarrow \infty}V_{n}(T,\vp,a)\leq J^{\xi}(T,\vp,a)+\eps.
\end{equation}
Since $\xi$ and $\eps$ are arbitrary, we have the desired result.

\subsection{Proof of Proposition~\ref{prop:dp}}\label{sec:proofofprop31}
\textbf{Step 1}. First we will show that $U$ is a fixed point of $\G$. Since $V_n \ge U$,
monotonicity of $\G$ implies that
\begin{multline*}
V_{n+1}(T,\vp,a)\geq \inf_{\tau \in \mathcal{S}(T)}\E^{\vp,a}\Bigl[\int_0^{\tau}\e^{-\rho
s}c(P_s) \, ds + \sum_{k} \e^{-\rho \sigma_k}1_{\{ \sigma_k \le \tau\}}K( (Y_k -
P_{\sigma_k-})_+) \\ + \e^{-\rho \tau}\M U(T-\tau, \vP_{\tau},P_{\tau-})\Bigr].
\end{multline*}
Taking the limit of the left-hand-side with respect to $n$ we obtain
\begin{multline*}
U(T,\vp,a) \geq \inf_{\tau \in \s(T)} \E^{\vp,a} \Bigl[ \int_0^\tau \e^{-\rho s} c(P_s) \, ds +
\sum_{k} \e^{-\rho \sigma_k}1_{\{ \sigma_k \le \tau\}}K( (Y_k - P_{\sigma_k-})_+) \\ +
\e^{-\rho \tau} \M U(T-\tau, \vP_{\tau}, P_{\tau-}) \Bigr].
\end{multline*}
Next, we will obtain the reverse inequality. Let $\tilde{\tau} \in \mathcal{S}(T)$ be an
$\eps$-optimal stopping time for the optimal stopping problem in the definition of $\G U$,
i.e.,
\begin{equation}\label{eq:eps-opti}
\begin{split}
\E^{\vp,a}&\left[\int_0^{\tilde{\tau}}\e^{-\rho s}c(P_s)\,ds + \sum_{k} \e^{-\rho
\sigma_k}1_{\{ \sigma_k \le \tilde{\tau}\}}K( (Y_k - P_{\sigma_k-})_+) + \e^{-\rho
\tilde{\tau}}\M U(T-\tilde{\tau}, \vP_{\tilde{\tau}},P_{\tilde{\tau}-})\right]
\\& \leq \G U(T,\vp,a)+ \eps.
\end{split}
\end{equation}
On the other hand, as a result of Proposition~\ref{prop:vnconvU} and the monotone convergence theorem
\begin{equation}\label{eq:lim-argu}
\begin{split}
U(T,\vp,a) & =\lim_{n \rightarrow \infty} V_{n}(T,\vp,a) \\  \leq \lim_{n \rightarrow
\infty}\E^{\vp,a} & \left[ \int_0^{\tilde{\tau}} \e^{-\rho s} c(P_s) \, ds + \sum_{k} \e^{-\rho
\sigma_k}1_{\{ \sigma_k \le \tilde{\tau}\}}K( (Y_k - P_{\sigma_k-})_+) + \e^{-\rho
\tilde{\tau}} \M V_{n-1}(T-\tilde{\tau}, \vP_{\tilde{\tau}}, P_{\tilde{\tau}-}) \right]
\\=\E^{\vp,a}&\left[\int_0^{\tilde{\tau}}\e^{-\rho s}c(P_s)\,ds + \sum_{k} \e^{-\rho
\sigma_k}1_{\{ \sigma_k \le \tilde{\tau}\}}K( (Y_k - P_{\sigma_k-})_+)  + \e^{-\rho
\tilde{\tau}}\M U(T-\tilde{\tau}, \vP_{\tilde{\tau}},P_{\tilde{\tau}-})\right] .
\end{split}
\end{equation}
Now, (\ref{eq:eps-opti}) and (\ref{eq:lim-argu}) together yield the desired result since $\eps$ is arbitrary.

\textbf{Step 2}. Let $\tilde{U}$ be another fixed point of $\G$ satisfying $\tilde{U} \leq
U_0=V_0$. Then an induction argument shows that $\tilde{U}\leq U$: Assume that $\tilde{U} \leq
V_n$. Then $\G \tilde{U} \leq \G V_n=V_{n+1}$, by the monotonicity of $\G$. Therefore for all
$n$, $\tilde{U} \leq V_n$, which implies that $\tilde{U} \leq \sup_n V_n =U$.
\hfill $\square$

\bibliographystyle{siam}

\bibliography{tracking-references
}

\begin{thebibliography}{10}

\bibitem{AllamDufourBertrand01}
{\sc S.~Allam, F.~Dufour, and P.~Bertrand}, {\em Discrete-time estimation of a
  {M}arkov chain with marked point process observations. {A}pplication to
  {M}arkovian jump filtering}, IEEE Trans. Automat. Control, 46 (2001),
  pp.~903--908.

\bibitem{Arjas92}
{\sc E.~Arjas, P.~Haara, and I.~Norros}, {\em Filtering the histories of a
  partially observed marked point process}, Stochastic Process. Appl., 40
  (1992), pp.~225--250.

\bibitem{AvivPazgal05}
{\sc Y.~Aviv and A.~Pazgal}, {\em A partially observed markov decision process
  for dynamic pricing}, Manage. Sci., 51 (2005), pp.~1400--1416.

\bibitem{Azoury85}
{\sc K.~S. Azoury}, {\em Bayes solution to dynamic inventory models under
  unknown demand distribution}, Management Sci., 31 (1985), pp.~1150--1160.

\bibitem{baylud08}
{\sc E.~Bayraktar and M.~Ludkovski}, {\em Sequential tracking of a hidden
  markov chain using point process observations}, To appear in Stochastic
  Processes and Their Applications,  (2008).

\bibitem{BS06}
{\sc E.~Bayraktar and S.~Sezer}, {\em Quickest detection for a {P}oisson
  process with a phase-type change-time distribution}, tech. rep., University
  of Michigan, 2006.

\bibitem{BensoussanMinjarez08}
{\sc A.~Bensoussan, M.~{\c{C}}akanyildirim, J.~Minjarez-Sosa, A.~Royal, and
  S.~Sethi}, {\em Inventory problems with partially observed demands and lost
  sales}, Journal of Optimization Theory and Applications, 136 (2008).

\bibitem{BensoussanSethiCR05}
{\sc A.~Bensoussan, M.~{\c{C}}akany{\i}ld{\i}r{\i}m, and S.~P. Sethi}, {\em On
  the optimal control of partially observed inventory systems}, C. R. Math.
  Acad. Sci. Paris, 341 (2005), pp.~419--426.

\bibitem{BensoussanMOR07}
\leavevmode\vrule height 2pt depth -1.6pt width 23pt, {\em A multiperiod
  newsvendor problem with partially observed demand}, Math. Oper. Res., 32
  (2007), pp.~322--344.

\bibitem{BensoussanEtalSICON07}
{\sc A.~Bensoussan, M.~C\c{}akanyildirim, and S.~P. Sethi}, {\em Partially
  observed inventory systems: The case of zero-balance walk}, SIAM J. Control
  Optim., 46 (2007), pp.~176--209.

\bibitem{BensoussanLiuSethi05}
{\sc A.~Bensoussan, R.~H. Liu, and S.~P. Sethi}, {\em Optimality of an
  {$(s,S)$} policy with compound {P}oisson and diffusion demands: a
  quasi-variational inequalities approach}, SIAM J. Control Optim., 44 (2005),
  pp.~1650--1676 (electronic).

\bibitem{BertsekasBookVol1}
{\sc D.~P. Bertsekas}, {\em Dynamic programming and optimal control. {V}ol.
  {I}}, Athena Scientific, Belmont, MA, third~ed., 2005.

\bibitem{BeyerSethi05}
{\sc D.~Beyer and S.~P. Sethi}, {\em Average cost optimality in inventory
  models with {M}arkovian demands and lost sales}, in Analysis, control and
  optimization of complex dynamic systems, vol.~4 of GERAD 25th Anniv. Ser.,
  Springer, New York, 2005, pp.~3--23.

\bibitem{bremaud}
{\sc P.~Bremaud}, {\em Point Processes and Queues}, Springer, New York, 1981.

\bibitem{CostaDavis89}
{\sc O.~L.~V. Costa and M.~H.~A. Davis}, {\em Impulse control of
  piecewise-deterministic processes}, Math. Control Signals Systems, 2 (1989),
  pp.~187--206.

\bibitem{DarrochMorris}
{\sc J.~N. Darroch and K.~W. Morris}, {\em Passage-time generating functions
  for continuous-time finite {M}arkov chains}, Journal of Applied Probability,
  5 (1968), pp.~414--426.

\bibitem{davis93}
{\sc M.~H.~A. Davis}, {\em Markov Models and Optimization}, Chapman \& Hall,
  London, 1993.

\bibitem{ElliottBook}
{\sc R.~J. Elliott, L.~Aggoun, and J.~B. Moore}, {\em Hidden {M}arkov models},
  vol.~29 of Applications of Mathematics (New York), Springer-Verlag, New York,
  1995.
\newblock Estimation and control.

\bibitem{ElliottMalcolm04}
{\sc R.~J. Elliott and W.~P. Malcolm}, {\em Robust {$M$}-ary detection filters
  and smoothers for continuous-time jump {M}arkov systems}, IEEE Trans.
  Automat. Control, 49 (2004), pp.~1046--1055.

\bibitem{ElliottMalcolm05}
\leavevmode\vrule height 2pt depth -1.6pt width 23pt, {\em General smoothing
  formulas for {M}arkov-modulated {P}oisson observations}, IEEE Trans. Automat.
  Control, 50 (2005), pp.~1123--1134.

\bibitem{MR1150206}
{\sc D.~Gatarek}, {\em Optimality conditions for impulsive control of
  piecewise-deterministic processes}, Math. Control Signals Systems, 5 (1992),
  pp.~217--232.

\bibitem{MR995463}
{\sc O.~Hern{\'a}ndez-Lerma}, {\em Adaptive {M}arkov control processes},
  vol.~79 of Applied Mathematical Sciences, Springer-Verlag, New York, 1989.

\bibitem{KarlinTaylor2}
{\sc S.~Karlin and H.~M. Taylor}, {\em A second course in stochastic
  processes}, Academic Press, New York, 1981.

\bibitem{LarivierePorteus}
{\sc M.~A. Lariviere and E.~L. Porteus}, {\em Stalking information: Bayesian
  inventory management with unobserved lost sales}, Manage. Sci., 45 (1999),
  pp.~346--363.

\bibitem{Lovejoy90}
{\sc W.~S. Lovejoy}, {\em Myopic policies for some inventory models with
  uncertain demand distributions}, Manage. Sci., 36 (1990), pp.~724--738.

\bibitem{LS07}
{\sc M.~Ludkovski and S.~Sezer}, {\em Finite horizon decision timing with
  partially observable {P}oisson processes}, tech. rep., University of
  Michigan, 2007.

\bibitem{Neuts}
{\sc M.~F. Neuts}, {\em Structured Stochastic Matrices of M/G/1 Type and Their
  Applications}, Marcel Dekker, New York, 1989.

\bibitem{SethiCheng97}
{\sc S.~P. Sethi and F.~Cheng}, {\em Optimality of {$(s,S)$} policies in
  inventory models with {M}arkovian demand}, Oper. Res., 45 (1997),
  pp.~931--939.

\bibitem{SongZipkin}
{\sc J.-S. Song and P.~Zipkin}, {\em Inventory control in a fluctuating demand
  environment}, Oper. Res., 41 (1993), pp.~351--370.

\bibitem{TreharneSox}
{\sc J.~T. Treharne and C.~R. Sox}, {\em Adaptive inventory control for
  nonstationary demand and partial information}, Manage. Sci., 48 (2002),
  pp.~607--624.

\bibitem{zab83}
{\sc J.~Zabcyzk}, {\em Stopping problems in stochastic control}, Proceedings of
  the International Congress of Mathematicians,  (1983), pp.~1425--1437.

\end{thebibliography}
\end{document}